\documentclass[reqno]{amsart}

\usepackage{enumitem}
\usepackage{amssymb}
\usepackage{hyperref}

\newcommand*{\mailto}[1]{\href{mailto:#1}{\nolinkurl{#1}}}
\newcommand{\arxiv}[1]{\href{http://arxiv.org/abs/#1}{arXiv:#1}}

\newtheorem{theorem}{Theorem}[section]
\newtheorem{definition}[theorem]{Definition}
\newtheorem{lemma}[theorem]{Lemma}
\newtheorem{example}[theorem]{Example}
\newtheorem{proposition}[theorem]{Proposition}
\newtheorem{corollary}[theorem]{Corollary}
\newtheorem{hypothesis}[theorem]{Hypothesis}
\newtheorem{remark}[theorem]{Remark}

\newcommand{\R}{{\mathbb R}}
\newcommand{\N}{{\mathbb N}}
\newcommand{\Z}{{\mathbb Z}}
\newcommand{\C}{{\mathbb C}}

\newcommand{\spr}[2]{( #1 , #2 )}

\newcommand{\be}{\begin{equation}}
\newcommand{\ee}{\end{equation}}

\newcommand{\comp}{\mathrm{c}}

\newcommand{\gD}{\mathfrak{D}}
\newcommand{\gH}{\mathfrak{H}}
\newcommand{\gB}{\beta}
\newcommand{\gA}{\alpha}

\newcommand{\gt}{\mathfrak{t}}
\newcommand{\gq}{\mathfrak{q}}
\newcommand{\gb}{\mathfrak{b}}
\newcommand{\gd}{{d}}
\newcommand{\rH}{\mathrm{H}}
\newcommand{\rL}{\mathrm{L}}
\newcommand{\rN}{\mathrm{N}}
\newcommand{\rD}{\mathrm{D}}

\newcommand{\im}{\mathrm{Im}}

\newcommand{\dom}[1]{\mathrm{dom}(#1)}
\newcommand{\loc}{\mathrm{loc}}
\newcommand{\ran}[1]{\mathrm{ran}(#1)}

\newcommand{\supp}[1]{\mathrm{supp}(#1)}

\newcommand{\qd}{{[1]}}
\newcommand{\Llocp}{L^1_{\mathrm{loc}}((a,b);dP)}
\newcommand{\Lab}{L^2(a,b)}

\newcommand{\eps}{\varepsilon}

\newcommand{\Sig}{\Sigma}

\numberwithin{equation}{section}

\begin{document}

\title[$\delta'$-interactions on Cantor-type sets]{One-dimensional Schr\"odinger operators\\ with $\delta'$-interactions on Cantor-type sets}

\author[J.\ Eckhardt]{Jonathan Eckhardt}
\address{School of Computer Science \& Informatics\\ Cardiff University\\
5 The Parade\\ Cardiff CF24 \\ UK}
\email{\mailto{j.eckhardt@cs.cardiff.ac.uk}}

\author[A.\ Kostenko]{Aleksey Kostenko}
\address{Faculty of Mathematics\\ University of Vienna\\
Oskar-Morgenstern-Platz 1\\ 1090 Wien\\ Austria}
\email{\mailto{duzer80@gmail.com}; \mailto{Oleksiy.Kostenko@univie.ac.at}}

\author[M.\ Malamud]{Mark Malamud}
\address{Institute of Applied Mathematics and Mechanics\\
NAS of Ukraine\\ R. Luxemburg str. 74\\
Donetsk 83114\\ Ukraine}
\email{\mailto{mmm@telenet.dn.ua}}

\author[G.\ Teschl]{Gerald Teschl}
\address{Faculty of Mathematics\\ University of Vienna\\
Oskar-Morgenstern-Platz 1\\ 1090 Wien\\ Austria\\ and International Erwin Schr\"odinger
Institute for Mathematical Physics\\ Boltzmanngasse 9\\ 1090 Wien\\ Austria}
\email{\mailto{Gerald.Teschl@univie.ac.at}}
\urladdr{\url{http://www.mat.univie.ac.at/~gerald/}}

\thanks{J. Differential Equations {\bf 257}, 415--449 (2014)}
\thanks{{\it Research supported by the Austrian Science Fund (FWF) under Grant No.\ Y330 and M1309.}}


\begin{abstract}
We introduce a novel approach for defining a $\delta'$-interaction on a subset of the real line of Lebesgue measure zero which is based
on Sturm--Liouville differential expression with measure coefficients. This enables us to establish basic spectral properties
(e.g.,  self-adjointness, lower semiboundedness and spectral asymptotics) of Hamiltonians with $\delta'$-interactions concentrated
on sets of complicated structures.
\end{abstract}

\maketitle

\section{Introduction}\label{sec:intro}

The main object of the present paper is the Hamiltonian $\rH$ in $L^2(a,b)$  associated with the differential expression
\be\label{eq:DE}
 -\frac{d}{dx} \frac{d}{dP(x)}+q(x),
\ee
where $P$ is a real-valued function on some interval $(a,b)\subseteq\R$ which is locally of bounded variation and $q \in L^1_{\loc}(a,b)$ is a real-valued function.
More specifically, we are only interested in the case when
\be\label{eq:P}
 P(x) = x + \nu(x), 
\ee
such that the Borel measure $d\nu$ is singular with respect to the Lebesgue measure.

Such kind of Sturm--Liouville operators with measure coefficients have a long history and for further details we refer to the monographs \cite{aghh}, \cite{fle}, \cite{ming}  and to the more recent papers \cite{ben}, \cite{bare}, \cite{measureSL}, \cite{EGNT}, \cite{EGNT2}, \cite{EGNT3}, \cite{SavShk99}, \cite{SavShk03}, \cite{volkmer2005}.
Let us only mention two particular examples: the Krein string operator \cite{kakr74}
\be
 -\frac{d}{dM(x)} \frac{d}{dx},
\ee
where $M$ is a nondecreasing function, and the Schr\"odinger operator with a measure potential (see, e.g., \cite{bare}, \cite{SavShk03})
\be\label{eq:Tdip}
 -\frac{d^2}{dx^2}+ dQ(x).
\ee
In particular, if the potential $dQ$ is a discrete measure, that is, $dQ(x)=\sum_k \gA_k\delta(x-x_k)$, then the differential expression in~\eqref{eq:Tdip} describes a $\delta$-interaction on the discrete set $X=\{x_k\}$ of strength $\gA$.

Similarly, if $q\equiv 0 $ and we set $d\nu(x)=\gB\delta(x)$ in \eqref{eq:P}, then the maximal operator associated with \eqref{eq:DE} in $L^2(\R)$ is given by (see Example \ref{ex:delta'} for details)
\begin{align}\label{eq:HBsingle}
\rH f & =-f'', & \dom{\rH} & = \left\lbrace f\in W^{2,2}(\R\backslash\{0\}): \begin{array}{c}
f'(0+)=f'(0-)\\
f(0+)-f(0-)=\gB f'(0+)
\end{array}\right\rbrace.
\end{align}
Hence, this operator describes a $\delta'$-interaction at $x=0$ of strength $\gB$ (see \cite{aghh}) and is formally given by
\be\label{eq:Hbeta}
\rH = -\frac{d^2}{dx^2}+\gB(\,\cdot\,,\delta')\delta'.
\ee
The existence of the
model \eqref{eq:Hbeta} was pointed out in 1980 by Grossmann, Hoegh--Krohn and Mebkhout
\cite{ghm}. However, the first rigorous treatment of \eqref{eq:Hbeta} was made by Gesztesy and
Holden in \cite{gh} using the method of boundary conditions. An alternate approach based on
generalized derivatives in the sense of distributions (using test functions which are allowed to
be discontinuous at the position of the interaction) was given in \cite{kur}.
However, again this approach only works for discrete supports.

One of the most traditional approaches to study Hamiltonians with $\delta'$-interactions is the method of boundary conditions (see, e.g.,  \cite{aghh}, \cite{KM_09}, \cite{KosMal10}). Note that only recently \cite{KosMal12} it was realized how to apply the form approach to investigate spectral properties of these Hamiltonians.
More precisely (see \cite{KosMal12} as well as \cite{BehLanLot12}),
a $\delta'$-interaction can be considered as a form sum of two forms $\gt_\rN$ and $\gb$, where
\begin{align}\label{eq:3.5}
\gt_\rN[f] & =\int_{\R} |f'(x)|^2\, dx, & \dom{\gt_\rN} & =W^{1,2}(\R\backslash\{0\}),
\end{align}
and
\begin{align}\label{eq:3.6}
\gb[f] & = \frac{|f(0+)-f(0-)|^2}{\gB}, & \dom{\gb} & =W^{1,2}(\R\backslash\{0\}).
\end{align}
Let us note that the operator
\begin{align}
\rH_\rN f & =- f'', & \dom{\rH_\rN} & =\{f\in W^{2,2}(\R\backslash \{0\}): f'(0+)=f'(0-)=0\},
\end{align}
is associated with the form $\gt_\rN$. Clearly, $\rH_\rN$ is the direct sum of Neumann realizations of $-\frac{d^2}{dx^2}$ in $L^2(\R_-)$ and $L^2(\R_+)$, respectively. Note that the form $\gb$ is infinitesimally form bounded with respect to the form $\gt_\rN$ and hence, by the KLMN theorem,  the form
\begin{align}\label{eq:3.7}
\gt[f] & =\gt_\rN[f]+\gb[f], & \dom{\gt} & =W^{1,2}(\R\backslash\{0\}),
\end{align}
is closed, lower semibounded and gives rise to the self-adjoint operator \eqref{eq:HBsingle}.

Let us also mention that the approximation by Schr\"odinger operators with scaled smooth potentials does not work for $\delta'$-interactions (see the details in  \cite[Appendix K]{aghh} and also \cite{cs}, \cite{enz}, \cite{GolHry}, \cite{GolMan}). All this shows that Hamiltonians with $\delta$ and $\delta'$-interactions are quite different. In particular, it is straightforward to introduce a $\delta$-interaction on an arbitrary set of Lebesgue measure zero.  To this end, one just needs to take an appropriate singular measure $dQ$ in \eqref{eq:Tdip}.
However, the situation with $\delta'$-interactions is quite unclear. To the best of our knowledge, only a few papers are devoted to the study of Hamiltonians with $\delta'$-interactions supported on sets of Lebesgue measure zero (see \cite{albniz}, \cite{braniz}, \cite{niz}). In \cite{albniz} and \cite{niz}, a $\delta'$-interaction on a compact set of Lebesgue measure zero is introduced with the help of boundary conditions. In the more recent paper \cite{braniz}, an abstract definition is given. Our main aim is to provide another definition of a $\delta'$-interaction using the generalized differential expression \eqref{eq:DE}. More precisely, based on the recent investigation of Sturm--Liouville operators with measure coefficients \cite{measureSL}, we show that the Hamiltonian with a $\delta'$-interaction supported on a closed set $\Sigma\subset \R$ of Lebesgue measure zero can be treated as a Sturm--Liouville operator \eqref{eq:DE} with a singular density \eqref{eq:P}. In this case, the singular measure $d\nu$ supported on $\Sigma$ is the strength of the $\delta'$-interaction. A precise definition will be given in Section~\ref{sec:def}. Moreover, we show that in the cases when either $d\nu$ is a discrete measure or $d\nu$ is a finite signed measure such that $\Sigma=\supp{d\nu}$ is a compact subset of $\R$ of Lebesgue measure zero, our definition coincides with the one in \cite{aghh}, \cite{gh} and in \cite{albniz}, \cite{niz}, respectively.

Our proposed approach has several advantages: First of all, it enables us to apply the well-developed Sturm--Liouville theory for the study of deficiency indices to Hamiltonians with $\delta'$-interactions (see Section~\ref{sec:sa}).
 In particular, we may describe the deficiency indices using Weyl's limit-point/limit-circle classification of the endpoints.
 The latter enables us to prove a simple self-adjointness criterion in terms of the interaction's strength $d\nu$ if $q\in L^\infty(a,b)$ (Theorem \ref{th:sa}).
 In Section \ref{sec:lsb}, we investigate lower semiboundedness of Hamiltonians with $\delta'$-interactions.
 First of all, Theorem \ref{th:gpw} extends the classical Glazman--Povzner--Wienholtz theorem to the case of Hamiltonians \eqref{eq:DE}--\eqref{eq:P} (the case of a discrete measure $d\nu$ was studied in \cite{KosMal12}).
 Moreover, we  introduce the quadratic form \eqref{eq:t01}--\eqref{eq:t02} which is associated with the Hamiltonian $\rH$ (see Lemma \ref{lem:3.2}).  
 This form plays a key role in the study of the negative spectrum of $\rH$.
 For instance, if $q\equiv 0$, then we show that the dimension of the negative subspace of $\rH$ equals the cardinality of $\Sig_-$, the topological support of the negative part $d\nu_-$ of $d\nu$ (Theorem \ref{th:kappa-}).
As a consequence of this result, let us mention the following fact: if the Hamiltonian $\rH$ is lower semibounded, then the negative part of $d\nu$ is a discrete measure (see Corollary \ref{cor:5.4}). Moreover, as in \cite{KosMal12} 
the form approach enables us to treat $\rH$ as a form perturbation of the Neumann realization $\rH_{\Sig}$ (see Section \ref{sec:VIII} for definitions) and then to investigate the discreteness of the spectrum of $\rH$. More precisely, using the discreteness criterion for the Neumann realization $\rH_\Sig$ (Theorem \ref{th:discr_N}), we obtain various necessary and sufficient conditions for the spectrum of $\rH$ to be discrete (see Sections \ref{ss:8.3} and \ref{ss:8.4}). In Section \ref{seq:res}, we approximate the Hamiltonian $\rH$ by Sturm--Liouville operators with smooth coefficients.
 Section \ref{sec:asymp} collects some results on spectral asymptotics of Hamiltonians with $\delta'$-interactions.
 Again, our definition immediately reveals the dependence of spectral asymptotics on the behavior of the measure $d\nu$ at the endpoints of the interval.
 Let us emphasize that in the regular case (that is, the interval $(a,b)$ is finite and $d\nu$ is a finite signed measure), the eigenvalues of $\rH$ admit the classical Weyl's asymptotic (Lemma \ref{lem:4.1}).
 However, the high energy behavior of the $m$-function as well as of the corresponding spectral function depends on the interaction strength $d\nu$ (see Theorem  \ref{th:4.2}).

\section{A Schr\"odinger operator with non-local interactions}\label{sec:def}

 In this section, we will introduce a particular kind of Sturm--Liouville operators with measure coefficients.
 For further details regarding the notion of measure Sturm--Liouville equations and operators we refer the reader to \cite{measureSL}, \cite{kac}, \cite{kakr74}, \cite{ming}.

 To set the stage, consider the differential expression
\begin{align}\label{eq:tau}
 \tau = - \frac{d}{dx} \frac{d}{dP(x)} + q(x)
\end{align}
on some interval $(a,b)$, where $q=\overline{q}\in L^1_{\loc}(a,b)$ and $P$ is a real-valued function on $(a,b)$ which is locally of bounded variation.
The corresponding real-valued Borel measure will be denoted by $dP$.
We shall also assume that the following holds:

\begin{hypothesis}\label{hyp:01}
The absolutely continuous part of $dP$ (with respect to the Lebesgue measure) is the Lebesgue measure and therefore, we may write $P(x) = x + \nu(x)$ such that the measure $d\nu$ is singular (with respect to the Lebesgue measure).
\end{hypothesis}

Note that we could also allow $q$ to be a measure as in \cite{measureSL}. However, for the sake of simplicity and readability we refrain from doing so here. 

In order for $\tau f$ to make sense, it is at least necessary that the function $f$ belongs to the class $AC_\loc((a,b);dP)$ of functions which are locally absolutely continuous with respect to $dP$.
This class consists of all functions $f$ which are locally of bounded variation and such that the corresponding Borel measure $df$ is absolutely continuous with respect to $dP$.
In this case one has
\be\label{eq:1_qd}
  df(x) = f^\qd(x)dP(x)
\ee
for some (unique) $f^\qd\in\Llocp$,
which is the Radon--Nikod\'ym derivative of $df$ with respect to $dP$ and called the {\em quasi-derivative of} $f$.
The function $f^\qd$ is defined almost everywhere (with respect to the Lebesgue measure) and (in order for $\tau f$ to make sense) has to be locally absolutely continuous with respect to the Lebesgue measure, such that
\begin{align}
 \tau f(x) = - (f^\qd)'(x) + q(x) f(x)
\end{align}
is defined for almost all $x\in(a,b)$ with respect to the Lebesgue measure.

 Thus, the maximal domain on which the differential expression \eqref{eq:tau} can be defined, is given by (see \cite{measureSL})
\be\label{eq:dmax}
\gD=\{f\in AC_{\loc}((a,b);dP): f^{[1]}\in AC_{\loc}(a,b)\}.
\ee
Consequently, the differential expression $\tau$ gives rise to a maximally defined closed operator $\rH$ in the Hilbert space $\Lab$, given by
 \begin{align}\label{eq:lmax}
 \rH f & =\tau f, & \dom {\rH} & =\gD_{\max}=\{ f\in L^2(a,b): f\in\gD,\ \tau f\in L^2(a,b)\}.
 \end{align}
 Although there may be various representatives of a function $f\in\dom{\rH}$ in $\gD$, the function $\tau f$ (and in fact, also the first quasi-derivative) are independent of this choice.
 For definiteness, by default we will always choose the unique left-continuous representative of $f$ which may be discontinuous only in points of mass of $dP$.
 In particular, $f$ has at most countably many points of discontinuity and the respective right-hand limits will be denoted with
\begin{align}
f(x+)=\lim_{\varepsilon \downarrow 0}f(x+\varepsilon).
\end{align}
Since $dP(x)=dx+d\nu(x)$, equation \eqref{eq:1_qd} turns into
\be\label{eq:1_qd'}
  f(x) - f(y) = \int_y^x f^\qd(t)dt + \int_{[y,x)} f^\qd(t)\,d\nu(t), \quad x,\, y\in(a,b),~y<x.
 \ee
Note that the first summand  on the right-hand side is locally absolutely continuous, while the second one is singular with respect to the Lebesgue measure. 
Let $\Sigma_{\min}$ be a minimal (non-topological) support of the measure $d\nu$ and note that the Lebesgue measure of $\Sigma_{\min}$  equals zero.
Moreover, we can choose $\Sigma_{\min}$ such that for all $x\in \Sigma_{\min}^c:=(a,b)\backslash\Sigma_{\min}$  the second summand in \eqref{eq:1_qd'} is differentiable at $x$ and its derivative is zero.  
Due to the continuity of $f^{[1]}$, the derivative $f'$ exists for all $x\in \Sigma_{\min}^c$   with $f'(x) = f^{[1]}(x)$.
  Since $f^{[1]}$ is continuous on $(a,b)$, $f'$ is continuous on $\Sigma_{\min}^c$ as well and admits a continuous (and even locally absolutely continuous) continuation to $(a,b)$ (which coincides with $f^\qd$). We will often keep the notation $f'$ for this continuation.

\begin{example}[$\delta'$-interaction on a discrete set]\label{ex:delta'}

 Let $(a,b)=\R$, $q\equiv 0$ and $\supp{d\nu}=X$, where $X=\{x_k\}_{k=-\infty}^{\infty}$ with $x_k<x_{k+1}$ for every $k\in \Z$, and $x_k\to \pm\infty$ as $k\to \pm\infty$.
 Then the singular part $d\nu$ of $dP$ is of the form
 \be
 d\nu(x)=\sum_{k\in\Z} \beta_k\delta(x-x_k),
 \ee
 for some $\beta_k\in\R$. 
 For every $f\in\gD_{\max}$, the functions $f$ and $f'$ are clearly absolutely continuous on the intervals $[x_k,x_{k+1}]$ for all $k\in \Z$.
 Moreover, $f'(x)=f^{\qd}(x)$ for all $x\in \R\backslash X$ and, since $f^\qd\in AC_{\loc}(\R)$, we get
 \be\label{eq:2.8}
 f'(x_k+)=f'(x_k-)=f^\qd(x_k)=:f'(x_k),\quad k\in\Z.
 \ee
 Next, using \eqref{eq:1_qd}, we obtain the jump condition
 \be\label{eq:2.9}
 f(x_k+)-f(x_k-)=f^\qd(x_k)\beta_k=f'(x_k)\beta_k,\quad k\in\Z,
 \ee
and hence the domain of the maximal operator is given by
 \begin{align}\begin{split}
 \dom{\rH} & =\{f\in L^2(\R): f,f'\in AC([x_k,x_{k+1}]) \text{ for all } k\in\Z, \\ & \qquad\qquad\qquad\qquad f\ \text{satisfies}\ \eqref{eq:2.8} \text{ and } \eqref{eq:2.9},\ f''\in L^2(\R)\}.
 \end{split}\end{align}
  Therefore, the maximal operator $\rH$ describes $\delta'$-interactions at the points $x_k$ with strengths $\beta_k$ (see, e.g. \cite{aghh}), that is, in this case the Hamiltonian $\rH$ associated with $\tau$ can be identified with the formal differential expression
 \begin{align}
  \rH & = -\frac{d^2}{dx^2}+\sum_{k\in\Z}\beta_k \spr{\,\cdot\,}{\delta'_k}\delta'_k, & \delta'_k & =\delta'(x-x_k).
 \end{align}
 \end{example}

 \begin{example}[$\delta'$-interaction on a Cantor-type set]\label{ex:niz}

Let $\Sigma$ be a closed compact subset of $\R$ of Lebesgue measure zero and pick a Borel measure $\mu$ whose topological support is $\Sigma$ (e.g.\ the measure associated with $\Sigma$ considered as a time scale \cite{TimeScales}).
 Following \cite{albniz} (see also \cite[\S 6]{braniz}), the Hamiltonian $\rL$ with a $\delta'$-interaction of strength $\gB\in L^1(\R,d\mu)$ on $\Sigma$ is defined by
 $\rL f =-f''$
 on functions $f\in \dom{ \rL}$ which belong to $W^{2,2}(\R\backslash\Sigma)$ and, moreover, admit the following integral representation
 \begin{align}\label{eq:1_qdNew}
 f(x) - f(y) & = \int_y^x f'(t)dt +\int_{[y,x)} g(t)d\mu(t), & f'(x) - f'(y) & = \int_y^x f''(t)dt,
 \end{align}
 where $g\in L^1(\R,d\mu)$ such that
 \begin{align}
 g(x)=\gB(x)f'(x),\quad x\in \Sigma.
 \end{align}
 Clearly, setting $d\nu(x) = \gB(x)d\mu(x)$ as well as $q\equiv 0$, we see that such an $f$ is locally absolutely continuous with respect to $dP$.
 Moreover, the second equality in \eqref{eq:1_qdNew} means that $f^\qd = f'$ is locally absolutely continuous on $\R$ and hence implies $f\in\dom{\rH}$.
Therefore, the operator $\rL$ coincides with the maximal operator $\rH$.
 \end{example}

\section{Self-adjointness}\label{sec:sa}

It is known that the adjoint $\rH_{\min}:=\rH^*$ of the maximal operator $\rH=\rH_{\max}$ defined in Section~\ref{sec:def} is symmetric in $L^2(a,b)$; see \cite[Theorem~4.4]{measureSL}. Moreover, $\rH_{\min}$ can be defined as the closure of the operator $\rH^0$ defined by $\tau$ on the domain
\be\label{eq:dom0}
\dom{\rH^0}:=\dom{\rH}\cap L^2_c(a,b), 
\ee
where $L^2_c(a,b)$ denotes the space of square integrable functions with compact support. 
In order to describe the  deficiency indices of $\rH_{\min}$ we need the following useful definition.

\begin{definition}\label{def:lplc}
We say that {\em $\tau$ is in the limit-circle (l.c.) case at $a$} (at $b$), if for each $z\in \C$ all solutions of $(\tau - z)u=0$ lie in $L^2(a,b)$ near a (near $b$). Furthermore, we say {\em $\tau$ is in the limit-point (l.p.) case at $a$} (at $b$) if for each $z\in \C$ there is some solution of $(\tau - z)u=0$ which does not lie in $L^2(a,b)$ near $a$ (near $b$).
\end{definition}

 At this point, let us mention that for every $z\in\C$ the differential equation $(\tau - z)u = 0$ admits precisely two linearly independent solutions; see \cite[Section~3]{measureSL}.

 The next result is the extension of the classical Weyl classification of deficiency indices of the operator $\rH_{\min}$ (cf. \cite[Theorem 5.2]{measureSL}).

 \begin{theorem}\label{th:weyl}
 Each boundary point is either in the l.p.\ case or in the l.c.\ case. Moreover, the deficiency indices of the operator $\rH_{\min}$ are $(n,n)$, where $n\in\{0,1,2\}$ is the number of boundary points which are in the l.c.\ case. In particular, the maximal operator $\rH$ is self-adjoint if and only if both boundary points are in the l.p.\ case.
 \end{theorem}

Theorem \ref{th:weyl} provides a powerful tool to investigate the self-adjointness of the operator $\rH$.
In particular, in the case when $q\in L^\infty(a,b)+L^1_{\mathrm{c}}(a,b)$ we obtain the following simple self-adjointness criterion.
Here $L^\infty(a,b)$ denotes the set of bounded measurable functions and $L^1_{\mathrm{c}}(a,b)$ the integrable functions with compact support.

\begin{theorem}\label{th:sa}
 If $q\in L^\infty(a,b)+L^1_{\mathrm{c}}(a,b)$, then $\rH=\rH^*$ if and only if the following two conditions are satisfied:
  \begin{itemize}
   \item[(i)] For some $c\in (a,b)$ we have $\mathbf{1}\notin L^2(a,c)$ or $P\notin L^2(a,c)$.
   \item[(ii)] For some $c\in(a,b)$ we have  $\mathbf{1}\notin L^2(c,b)$ or $P\notin L^2(c,b)$.
 \end{itemize}
 \end{theorem}

\begin{proof}
 Since the self-adjointness is stable under bounded perturbations and perturbations with compact support do not change the behavior of solutions near the endpoints,
it suffices to prove the claim for the case $q\equiv 0$.
 However, in this case, the equation $\tau u=0$ has the two linearly independent solutions
\begin{align}
 u_1(x)=1 \quad \text{and} \quad u_2(x)=P(x)=x+\nu(x).
 \end{align}
 To complete the proof, it suffices to apply Theorem \ref{th:weyl}.
\end{proof}

 Clearly, we can also reformulate Theorem \ref{th:sa} as follows.

\begin{corollary}\label{cor:2.3}
 If $q\in L^\infty(a,b)+L^1_{\mathrm{c}}(a,b)$, then  $\rH\neq \rH^*$ if and only if at least one of the following two conditions is satisfied:
 \begin{itemize}
  \item[(i)] For some $c\in(a,b)$ we have $P\in L^2(a,c)$ and $a>-\infty$.
  \item[(ii)] For some $c\in(a,b)$ we have $P\in L^2(c,b)$ and $b < +\infty$.
 \end{itemize}
 \end{corollary}

 As an immediate consequence we obtain the following result.

 \begin{corollary}\label{cor:2.2}
 If $(a,b)=\R$ and $q\in L^\infty(\R)+L^1_{\mathrm{c}}(\R)$, then $\rH=\rH^*$.
 \end{corollary}

 \begin{remark}
 Note that Corollary \ref{cor:2.2} implies the self-adjointness of the Hamiltonians discussed in Examples \ref{ex:delta'} and \ref{ex:niz}. In particular, the self-adjointness of the one in Example \ref{ex:delta'} was established in \cite[Theorem 4.1]{bsw}. In the case when $\Sigma$ is a closed compact subset of $\R$ of Lebesgue measure zero and $\gB\in L^1(\R,d\mu)$, where $\mu$ is a Radon measure on $\Sigma$ (or equivalently $d\nu(x)=\gB(x)d\mu(x)$ is a finite measure on $\Sigma$), the self-adjointness of the Hamiltonian in Example \ref{ex:niz} was proved in \cite{albniz} and \cite{braniz}. Let us mention that Corollary \ref{cor:2.2} implies the self-adjointness of this Hamiltonian under Hypothesis \ref{hyp:01}
 \end{remark}

\section{Lower semiboundedness}\label{sec:lsb}

\subsection{A Glazman--Povzner--Wienholtz type theorem}\label{ss:GPW}

Let $(a,b)=\R_+$ and assume that $\tau$ is regular at $x=0$, i.e., $q\in L^1(0,c)$ and $|\nu|((0,c))<\infty$ for all $c>0$. Consider the restricted operator subject to the Dirichlet condition at $x=0$:
\begin{align}\label{eq:L_D}
\rH_{\rD} & = \rH\upharpoonright \dom{\rH_\rD}, & \dom{\rH_\rD} & =\{f\in\dom{\rH}: f(0)=0\},
\end{align}
where $\rH$ and $\dom{\rH}$ are given by \eqref{eq:lmax}.

\begin{theorem}\label{th:gpw}
Let $(a,b)=\R_+$ and assume that the topological support $\supp{d\nu}=\Sigma\subset \R_+$ satisfies at least one of the following conditions:
\begin{itemize}
\item[(i)] $\Sig$ is bounded.
\item[(ii)] $\Sig$ has Lebesgue measure zero, $|\Sig|=0$.
\item[(iii)] $\Sig$ contains an infinite number
of gaps near $+\infty$ whose lengths do not tend to zero.
\end{itemize}
If the minimal operator $\rH_{\min}=\rH_\rD^*$ acting in $L^2(\R_+)$ is lower semibounded, then it is self-adjoint.
\end{theorem}

\begin{proof}
(i) If $\Sig$ is a bounded subset of $\R_+$, then the claim immediately follows from the classical Glazman--Povzner--Wienholtz theorem by employing the Glazman separation principle. 

(ii) Assume that $|\Sig|=0$. Without loss of generality we can assume that $\rH_\rD^*\ge I$. If our differential equation was in the limit-circle case at $+\infty$, then
 we could find a nontrivial square integrable solution $u$ of $\tau u =0$ with $u(0) = 0$.
 By our assumption on $\Sigma$, for every $n\in\N$ the set $\Sig_n:=\Sig\cap [n,n+1]$ is closed and $|\Sig_n|=0$. Therefore, fix $\varepsilon \in (0,1)$ and choose an open $\varepsilon$-neighborhood $B_\varepsilon(\Sig_n)$ of $\Sig_n$. Now choose functions $\varphi_n\in C^2(\R_+)$ such that  
\begin{align}\label{eq:phi1}
\begin{split}
&0  \le \varphi_n\le 1,\qquad  -\frac{2}{1-\eps}  \le \varphi_n'\le 0\\
 \varphi_n(x)  =&
\begin{cases}
1,& x \le n,\\
0,& x\in [n+1,+\infty),
\end{cases}
\qquad  \varphi_n'(x)=0\ \ \text{on}\ \ B_\varepsilon(\Sig_n),
\end{split}
\end{align}
and introduce
\begin{equation}
  u_n(x)=u(x)\varphi_n(x), \quad x\in\R_+.
\end{equation}
Clearly, the support of $u_n$ is contained in $[0,n+1]$. Let us show that $u_n\in\dom{\rH_\rD^*}$. Firstly, note that $u_n^\qd(x)=u^\qd(x)\varphi_n(x)+u(x)\varphi_n'(x)$ since $\varphi_n'=\varphi_n^\qd$ on $\R_+$. Therefore, $u_n^\qd\in AC_{\loc}(\R_+)$ since $\varphi_n'$ vanishes on $\Sig$. Hence in order to show that $\tau u_n\in L^2(\R_+)$, it suffices to note that $u_n$ equals $u$ near zero and vanishes near $+\infty$.

Furthermore, noting that $u$ is a solution of $\tau u = 0$, $\varphi_n'\equiv 0$ for all $x\notin (n,n+1)$, and $u^{[1]}(x)=u'(x)$ for all $x\in (n,n+1)$, we get
 \begin{align}\label{3.16}
 \begin{split}
(\rH u_n,u_n)&=\int_{\R_+} [-(u^{[1]}_n)' + qu_n]u_ndx = - \int_{\R_+} [u^\qd\varphi'_n+ u'\varphi'_n + u
\varphi''_n]u\,\varphi_n dx \\
&= - \int^{n+1}_{n} [2 u'\varphi'_n + u
\varphi''_n]u\varphi_n dx \\
&= -\frac{1}{2} \int^{n+1}_{n} \big[ \bigl(u^2\bigr)'\bigl(\varphi^2_n\bigr)'+ 2u^2
\varphi''_n \varphi_n \big]dx = 
\int^{n+1}_{n}u^2 \bigl(\varphi'_n\bigr)^2dx.
\end{split}
\end{align}
In summary we obtain
\begin{equation}\label{3.18}
\int^{n}_{0}u^2dx \le (u_n,u_n) \le (\rH u_n,u_n) \le  \frac{4}{(1-\eps)^2}\int^{n+1}_{n}u^2 dx.
\end{equation}
Noting that $u\in L^2(\R_+)$, inequality \eqref{3.18} implies that $u\equiv 0$. This contradiction completes the proof.

(iii) By the assumption on $\Sigma$ we can pick a sequence of points $x_j\to\infty$ and a positive number $\eps>0$ such that
$(x_j,x_j+\eps) \cap\Sigma=\emptyset$ for every $j\in\N$. Now choose functions $\varphi_j\in C^2(\R)$ such that
\begin{align}
0 & \le \varphi_j\le 1, & \varphi_j(x) & =
\begin{cases}
1,& x \le x_j,\\
0,& x\ge x_j+\eps,
\end{cases}
&
-\frac{2}{\eps} & \le \varphi_j'\le 0,
\end{align}
and introduce
\begin{equation}
  u_j(x)=u(x)\varphi_j(x), \quad x\in\R_+.
\end{equation}
Clearly, the support of $u_j$ is contained in $[0,x_j+\eps]$ and, moreover, $u_j\in\dom{\rH_\rD^*}$. The rest of the proof is analogous to the proof of (ii) and we leave it to the reader.
\end{proof}

\begin{remark}\label{rem:4.2}
\begin{itemize}
\item[(i)] Note that Theorem \ref{th:gpw} admits an obvious extension to whole line case.
\item[(ii)] In the case of $\delta'$-point interactions (that is, $\Sigma$ is a discrete set) Theorem \ref{th:gpw} was established in \cite{KosMal12}. Note also that condition (ii) can be easily extended to the case $|\Sig|<\infty$ or, more generally, $|\Sig\cap [n,n+1]|\le \varepsilon<1$ for all $n$ large enough.
\item[(iii)] Similar results for Hamiltonians with $\delta$-type interactions can be found in \cite{AKM_10} and \cite{hrymyk12}.
\end{itemize}
\end{remark}

\subsection{The quadratic form}\label{ss:qform}

Consider the following two forms in $L^2(\R)$
\begin{align}\label{eq:t01}
\gt_{0}^0[f] & =\int_{\R}|f^{[1]}|^2dP(x), & \gq[f] & =\int_{\R} q(x) |f|^2\,dx,
\end{align}
defined on the respective domains
\be\label{eq:t02}
\dom {\gt_{0}^0} = W^{1,2}_{\comp}(\R;dP) =\{f\in L^2_{\comp}(\R): f\in AC_\loc(\R;dP),\ f^{[1]}\in L^2(\R;|dP|)\}
\ee
and
\be\label{eq:t03}
\dom{\gq} =\{f\in L^2(\R): |\gq[f]|<\infty\}.
\ee
Hereby, note that the form $\gq$ is  lower semibounded (and hence closed) if so is $q$.
Let us introduce the form $\gt^0$ as a form sum of the two forms $\gt_{0}^0$ and $\gq$:
\begin{align}\label{eqnform}
\gt^0[f] & =\gt_{0}^0[f]+\gq[f], & \dom{\gt^0} & =\dom{\gt_{0}^0}\cap \dom{\gq}=\dom{\gt_{0}^0}.
\end{align}
The next result establishes a connection between the form $\gt^0$ and the operator $\rH$.

\begin{lemma}\label{lem:3.2}
\begin{itemize}
\item[(i)] If $f\in\dom {\rH^0}$, then $f\in\dom{\gt^0}$ with  $(\rH^0 f,f) = \gt^0[f]$.
\item[(ii)] Assume additionally that $\Sig$ satisfies at least one of the conditions (i)--(iii) of Theorem \ref{th:gpw}. If the form $\gt^0$ is lower semibounded, then it is closable and the operator associated with its closure $\gt=\overline{\gt^0}$ coincides with the self-adjoint operator $\rH$.
\end{itemize}
\end{lemma}
\begin{proof}
(i) Let $f\in\dom{\rH^0}=\dom {\rH}\cap L^2_{\comp}(\R)$ and integrate by parts to obtain
\begin{align}\begin{split}
(\rH f,f)& =(\rH^0f,f)=\int_{\R}\tau f(x) \overline{f(x)}\, dx\\
&=-\int_{\R} \overline{f(x)}\, df^\qd(x)+\gq[f]=\int_{\R} f^\qd(x)\, d\overline{f(x)}+\gq[f]\\
&=\int_{\R}|f^{\qd}(x)|^2dP(x)+\gq[f]=\gt^0_0[f]+\gq[f]=\gt^0[f].
\end{split}\end{align}

(ii) If the form $\gt^0$ is lower semibounded, then (i) implies that so is the operator $\rH^0$ and the form $\gt^0$ is closable.
Thus $\rH^0$ is essentially self-adjoint by Theorem \ref{th:gpw} (see also Remark \ref{rem:4.2} (i)). To complete the proof of (ii) it remains to note that $\dom{\rH^0}$ is a core for $\gt=\overline{\gt^0}$.
\end{proof}

\begin{remark}
If $d\nu$ is a discrete measure $d\nu(x)=\sum_{k=1}^\infty\beta_k\delta(x-x_k)$, then
\be
\int_{\R}|f^{\qd}|^2d\nu(x)=\sum_{k=1}^\infty\beta_k|f'(x_k)|^2=\sum_{k=1}^\infty\frac{|f(x_k+)-f(x_k-)|^2}{\beta_k}.
\ee
Moreover, in this case the domain of $\gt$ is $W^{1,2}(\R\backslash X)$, where $X=\{x_k\}_{k=1}^\infty$.
Using this form, the spectral properties (discreteness of the spectrum, essential spectrum, etc.) of the corresponding lower semibounded Hamiltonian $\rH=\rH_{X,\gB,q}$  were studied in great detail in \cite{KosMal12}.
\end{remark}

 \section{Negative spectrum}

 Let us recall the following fact, known as the Hahn decomposition (cf.\ \cite[Section 3.1]{bo}): For any signed Borel measure $d\nu$   there are two disjoint Borel sets $\Omega_+$, $\Omega_-$ such that $(a,b)=\Omega_+\cup\Omega_-$ and for any Borel set $E_\pm\subseteq \Omega_\pm$ it holds that $\pm d\nu(E_\pm)  \ge 0$.
 Moreover, if $(a,b)=\tilde{\Omega}_+\cup\tilde{\Omega}_-$ is another decomposition with this property, then $\Omega_\pm$ and $\tilde{\Omega}_\pm$ differ at most in a set of $|d\nu|$-measure zero.
 This decomposition is called {\em the Hahn decomposition} and for each Borel set $E\subseteq (a,b)$ we set
 \be\label{eq:jor}
  d\nu_\pm(E)= d\nu(E\cap\Omega_\pm).
 \ee
 Hereby, notice that
 \be
  d\nu(E)=d\nu_+(E)+d\nu_-(E),
 \ee
which is called {\em the Jordan decomposition} of $d\nu$. The measures $d\nu_+$ and $d\nu_-$ are called {\em the positive} and {\em the negative} part of $d\nu$, respectively. Finally, we introduce the following two quantities
 \be\label{eq:k_-}
 \kappa_-(d\nu)=\begin{cases}
\#\supp{d\nu_-}, & d\nu_- \ \text{is pure point},\\
\infty, & \text{otherwise},
\end{cases}
\ee
and for $\rH=\rH^*$
\be
\kappa_-(\rH)=\dim\ran{\chi_{(-\infty,0)}(\rH)}.
\ee
Note that $\kappa_-(\rH)$ is the number of negative eigenvalues of $\rH$ if $\kappa_-(\rH)$ is finite.

 We are now in the position to formulate the main result of this section.

 \begin{theorem}\label{th:kappa-}
 Let $q\equiv 0$ and $d\nu$ be a signed Borel measure on $\R$ which is singular with respect to the Lebesgue measure. If $\rH$ is the corresponding self-adjoint operator in $L^2(\R)$,  then
 \be\label{eqnkap}
 \kappa_-(\rH)=\kappa_-(d\nu).
 \ee
 \end{theorem}
 \begin{proof}

Firstly, assume that $\kappa_-(d\nu)$ is finite. This means that we may choose $\Sigma_-=\{x_k\}_{k=1}^N$, where $N=\kappa_-(d\nu)<\infty$. In particular, this implies that the Hamiltonian $\rH$ and hence the form $\gt$ are lower semibounded. Moreover, each $f\in \dom {\rH}$ satisfies the following jump condition at $x_k$:
 \begin{align}
 f'(x_k+) & =f'(x_k-), &  f(x_k+)-f(x_k-) & =\gB_k f'(x_k),
 \end{align}
 where $\gB_k = \nu(\{x_k\})<0$.

 Choose  $\varepsilon_k=N|\gB_k|>0 $ and define the following functions
 \be\label{eq:test_f}
 f_k^{[1]}(x)=\begin{cases}
 1,& x= x_k,\\
 \frac{1}{N},& x\in (x_k,x_k+\varepsilon_k]\backslash\Sigma,\\
 0, &x\in \Sigma\backslash \{x_k\},\\
 0,& x \in\R\backslash[x_k,x_k+\varepsilon_k],
 \end{cases}
 \ee
for $k\in\{1,\dots,N\}$.  Note that the functions
 \be\label{eq:test_f2}
 f_k(x) =\int_{x_k-1}^x f_k^{[1]}(t)dP(t)=\begin{cases}
 0,& x\in \R\backslash (x_k,x_k+\varepsilon_k],\\
 \frac{1}{N}(x-x_k-\varepsilon_k),& x \in (x_k,x_k+\varepsilon_k],
 \end{cases}
 \ee
 belong to $W^{1,2}(\R;|dP|)\cap L^2_{\comp}(\R)$ and hence also to $\dom{\gt}$.
 Moreover, for an arbitrary linear combination $f=\sum_{k=1}^Nc_kf_k$ (with $c_k\in\C$), we get
 \begin{align}\begin{split}
 \gt[f]&=\int_{\R}\Big|\sum_{k=1}^N\frac{1}{N}c_k\chi_{(x_k,x_k+\varepsilon_k)}(x)\Big|^2dx+\sum_{k=1}^N\gB_k|c_k|^2\\
 &< N\sum_{k=1}^N \frac{1}{N^2}\int_{x_k}^{x_k+N|\gB_k|}|c_k|^2dx-\sum_{k=1}^N |\gB_k||c_k|^2=0.
 \end{split}\end{align}
Note that the inequality is strict since the functions $f_k$, $k\in\{1,\dots,N\}$, are linearly independent. 
Hence we conclude that $\kappa_-(\rH)\ge N$. Since the converse inequality follows from the fact that $d\nu_-$ is a rank
$N$ perturbation, we arrive at \eqref{eqnkap}.

 It remains to prove the case when $\kappa_-(d\nu)=\infty$.
 Without loss of generality we can assume that our operator is lower semibounded. Indeed, if $\rH$ is not lower semibounded, then $\kappa_-(\rH)=\infty$ is obvious.

By definition, either $d\nu_-$ is pure point and supported on an infinite set or the singular continuous part of $d\nu_-$ is nontrivial. In the first case, we can prove the claim by using the above argument. Namely, using test functions \eqref{eq:test_f2}, \eqref{eq:test_f}, we can show that $\kappa_-(\rH)>n$ for any $n\in\N$. So, assume that the singular part of $d\nu_-$ is nontrivial. Denote by $\Sigma_{\min}$ and $\Sigma_{\min}^-$ minimal supports of $d\nu$ and $d\nu_-$, respectively. The latter means that both $\Sigma_{\min}$ and $\Sigma_{\min}^-$ have Lebesgue measure zero and $d\nu(\Omega\backslash\Sigma_{\min})=0$ and $d\nu_-(\Omega\backslash\Sigma_{\min}^-)=0$ for any measurable set $\Omega\subseteq \R$.
This, in particular, implies that there is a sequence $\{\Omega_k\}_{k=1}^\infty$ of bounded subsets of $\Sigma_{\min}^-$ such that
\begin{align}
\Omega_k & \subset \Sigma_{\min}^-, & d\nu_-(\Omega_k) & <0, &  |\Omega_k| & =0, & \Omega_i\cap\Omega_j & =\emptyset,\quad i\neq j.
\end{align}
Let also $\Omega_k\subset[x_k,y_k]\subset \R$ for every $k\in \N$.

Let $N\in \N$. Set $\varepsilon_k=N{|\gB_k|}$, where $\gB_k =d\nu(\Omega_k)=d\nu_-(\Omega_k)<0$ and define the functions
\be
 f_k(x) =\int_{x_k-1}^xf_k^{[1]}(t)dP(t), \quad x\in\R,
 \ee
 for every $k\in\N$, where
 \begin{align}
 f_k^{[1]}(x)=\begin{cases}
 1,& x\in \Omega_k,\\
 \frac{1}{N},& x\in (y_k,y_k+\varepsilon_k]\backslash\Sigma_{\min},\\
0, &x\in \Sigma_{\min}\cap (y_k,y_k+\varepsilon_k],\\
 0,& x \notin\Omega_k\cup(y_k,y_k+\varepsilon_k].
 \end{cases}
 \end{align}
 Clearly, $f_k\in W^{1,2}(\R,|dP|)\cap L^2_{\comp}(\R)$ and hence also $f_k\in \dom{\gt}$. Next, for any finite sequence $\{c_k\}_{k=1}^N$  set $f=\sum_{k=1}^N c_k f_{k}$. As in the first part of the proof, we get
 \begin{align}\begin{split}
  \gt[f]&=\int_{\R}\Big|\sum_{k=1}^N \frac{c_k}{N}\chi_{(y_{k},y_{k}+\varepsilon_{k})}(x)\Big|^2dx+\sum_{k=1}^N\gB_{k}|c_k|^2\\
 &< \frac{1}{N}\sum_{k=1}^N \int_{y_{k}}^{y_{k}+N|\beta_{k}|}|c_k|^2dx-\sum_{k=1}^N|\gB_{k}||c_{k}|^2 =0
\end{split}\end{align}
as before. Hence we conclude that $\kappa_-(\rH)\ge N$ and since $N\in\N$ is arbitrary we get $\kappa_-(\rH)=\infty$.
 \end{proof}

 \begin{remark}
In the case of $\delta'$-point interactions, Theorem \ref{th:kappa-} was established in \cite{GolOri10} and \cite{KosMal10} (see also \cite{KosMal12b} and \cite{niz} for further details). Let us also mention that under additional  restrictive assumptions on $\Sigma$, Theorem \ref{th:kappa-} was established in \cite{braniz} by employing a different approach.
 \end{remark}

  \begin{remark}
 In the case of a Schr\"odinger operator with $\delta$-interactions, the problem of estimating the number of negative eigenvalues is
rather nontrivial, even in the case of finitely many point interactions.  For further details we refer to, e.g., \cite{AKMN13}, \cite{albniz03}, \cite{GolOri10}, \cite{KM_09}, \cite{KosMal12b}, and \cite{ogu}.
\end{remark}
Using Theorem \ref{th:kappa-}, we can easily prove the following statement.

\begin{corollary}\label{cor:5.3}
In addition to the assumptions of Theorem \ref{th:kappa-}, suppose that $\Sigma_-$ is a bounded subset of $\R$. Then the operator $\rH$ is lower semibounded if and only if $\Sigma_-$ can be chosen finite. Moreover, if $\Sigma_-$ is infinite, then the negative part of the spectrum of $\rH$ is discrete.
\end{corollary}

 \begin{proof}
 Since $\Sigma$ is bounded, there is a bounded interval $(c,d)$ such that $\Sigma\subset (c,d)$. Arguing as in the proof of Theorem \ref{th:gpw}, the operator $\rH$ is a rank two perturbation (in the resolvent sense) of the orthogonal sum $\rH_{(-\infty,c)}\oplus \rH_{(c,d)}\oplus \rH_{(d,+\infty)}$ of restricted operators with Dirichlet boundary conditions at their endpoints. Note that the operators $\rH_{(-\infty,c)}$ and $\rH_{(d,+\infty)}$ are nonnegative. Moreover, the spectrum of the operator $\rH_{(c,d)}$ is discrete by \cite[Corollary~8.2]{measureSL}. %Arguing as in the proof of Theorem \ref{th:kappa-}, 
 We can show that the negative spectrum of $\rH_{(c,d)}$ is finite if and only if $\Sigma_-$ can be chosen finite (in order to show this one needs to consider a new measure $d\tilde{\nu}$ which coincide with $d\nu$ on $(c,d)$ and equals $0$ on $\R\setminus(c,d)$ and then to apply Theorem \ref{th:kappa-}).  Otherwise, the negative spectrum of $\rH_{(c,d)}$ is unbounded from below since it is discrete. Finally, noting that these spectral properties are stable under finite rank perturbations, the claim follows.
 \end{proof}

 \begin{remark}
Using another approach, Corollary \ref{cor:5.3} was established in \cite{braniz} under the additional assumptions that $\Sigma$ is a compact subset of $\R$ of Lebesgue measure zero.
 \end{remark}

 \begin{corollary}\label{cor:5.4}
If the operator $\rH$ is lower semibounded, then the negative part $d\nu_-$ of $d\nu$ is a discrete measure.
\end{corollary}

\begin{proof}
Assume that $d\nu_-$ is not discrete. Then there is a finite subinterval $(c,d)\subset \R$ such that $\# \big((c,d)\cap \Sigma_-\big)=\infty$.
The operator $\rH$ can be considered as a rank two perturbation of the orthogonal sum $\rH_{(-\infty,c)}\oplus \rH_{(c,d)}\oplus \rH_{(d,+\infty)}$ of restricted operators with Dirichlet boundary conditions at their endpoints. Arguing as in the proof of Corollary \ref{cor:5.3}, we can show that the operator $\rH_{(c,d)}$ is not bounded from below since its spectrum is discrete and $\kappa_-(\rH_{(c,d)})=\# \big((c,d)\cap \Sigma_-\big)=\infty$. Since lower semiboundedness is stable under finite rank perturbations, we conclude that $\rH$ is unbounded from below.
\end{proof}

 \begin{remark}
If $\kappa_-(d\nu)=\infty$, then one may try to prove Theorem \ref{th:kappa-} by approximating the form $\gt$ associated with the Hamiltonian $\rH$ by forms $\gt_n$, $n\in\N$, such that $\kappa_-(\gt_n)=\kappa_-(d\nu_n)=n$ and the corresponding measures $d\nu_n$ converge weakly-$*$ to $d\nu$. This proof clearly works in the case when $\gt$ is lower semibounded and the negative part $d\nu_-$ of $d\nu$ is discrete. However, if the negative part $d\nu_-$ has a nontrivial singular continuous component, then the latter is no longer true. First of all, for the form domains we get $\dom{\gt_n}\not\subseteq \dom{\gt}$ and, moreover, the test functions \eqref{eq:test_f2}, \eqref{eq:test_f} do not belong to $\dom{\gt}$ since the functions from $\dom{\gt}$ are continuous on $\Sigma_-$. On the other hand, by Corollary \ref{cor:5.4} the Hamiltonian $\rH$ is not lower semibounded in this case and hence we cannot deduce the information about $\kappa_-(\rH)$ from $\kappa_-(d\nu)$.
 \end{remark}

 \section{Approximation by Hamiltonians with smooth coefficients}\label{seq:res}

 In this section we restrict our considerations to the regular case, that is, when $(a,b)$ is bounded, the measure $|d\nu|$ is finite and the function $q$ is integrable.
 Moreover, for simplicity we will assume that our interval is the unit interval $(0,1)$.
 We consider the operator $\rH_\rD$, which is the restriction of the maximal operator $\rH$ subject to Dirichlet boundary conditions,
\begin{align}\label{eq:l_dir}
\rH_\rD f & =\tau f, & \dom{\rH_\rD} & =\{f\in L^2(0,1): f\in \gD_{\max},\ f(0)=f(1)=0\}.
\end{align}
It follows from \cite[Section 7]{measureSL} that the operator $\rH_\rD$ is self-adjoint.

\begin{theorem}\label{th:approx}
Let $\{d\nu_k\}_{k=1}^\infty$ be a sequence of finite measures on $(0,1)$.
\begin{itemize}
\item[(i)] If $d\nu_k$ converges to $d\nu$ in a weak-$*$ topology, such that for all $f\in C[0,1]$
\be\label{eq:3.2}
\int_{(0,1)} f(t)d\nu_k(t)\to \int_{(0,1)} f(t)d\nu(t),\quad  k\to\infty,
\ee
then there is a subsequence $\{d\nu_{k(j)}\}_{j=1}^\infty$ such that the corresponding operators $\rH_{\rD,k(j)}$ (with the same potential $q$) converge to $\rH_\rD$ in the norm resolvent sense.
\item[(ii)] If, in addition, $d\nu$ and all $d\nu_k$ are nonnegative measures satisfying \eqref{eq:3.2} for all $f\in C[0,1]$, then the corresponding operators $\rH_{\rD,k}$ converge to $\rH_\rD$ in the norm resolvent sense and, moreover, for all $n\in\N$
\be\label{eq:3.3}
\lambda_n(k)\to \lambda_n,\quad  k\to\infty.
\ee
\end{itemize}
\end{theorem}
\begin{proof}
 The resolvent of $\rH_{\rD,k}$ admits the representation (cf.\ \cite[Section 8]{measureSL})
\begin{align}\begin{split}
R_k(z)f(x):=(\rH_{\rD,k}-z)^{-1}f(x)=& \int_0^1 G_k(x,t;z)\, f(t)\, dt, \quad x\in(0,1),\\
G_k(x,t;z)&= \frac{1}{\psi_k(z,0)} \begin{cases}
\phi_k(z,t)\psi_k(z,x), & t\le x,\\
\phi_k(z,x)\psi_k(z,t), & t\ge x.
\end{cases}\label{eq:Gf}
\end{split}\end{align}
Here $\phi_k(z,\cdot\,)$ and $\psi_k(z,\cdot\,)$ are the solutions of $(\tau_{k} -z)u=0$ with the initial conditions $\phi_k(z,0)=\psi_k(z,1)=0$ and $\phi_k^{[1]}(z,0)=\psi_k^{[1]}(z,1)=1$.

Assume for simplicity that $q\equiv 0$ and
\be\label{eq:ker}
P(1) - P(0) =1+\nu(1) - \nu(0) \neq 0.
\ee
The latter means that $0\in\rho(\rH_\rD)$ since $\phi(0,x)=P(x)-P(0)$ as well as $\psi(0,x)=P(x)-P(1)$ and hence $W(\phi,\psi)(0)=P(1) - P(0)\neq 0$.
Without loss of generality, we can assume that $1+\nu_k(1) -\nu_k(0) \neq 0$ since $\lim_{k} \nu_k(1) - \nu_k(0)=\nu(1) - \nu(0)$ due to \eqref{eq:3.2}.
 Therefore, the inverse of $\rH_{\rD,k}$ is given by \eqref{eq:Gf} with $P_k$ in place of $P$.

Furthermore, from \eqref{eq:3.2} we conclude that there is a subsequence $P_j$ such that $P_j(x)\to P(x)$ for all but countably many $x\in [0,1]$ (see \cite[8.1.8 Proposition]{boII}).
As a consequence, the respective Green's functions $G_j(\,\cdot\,,\cdot\,;0)$ converge to $G(\,\cdot\,,\cdot\,;0)$ almost everywhere.
Since the Green's functions are uniformly bounded, we furthermore conclude that $R_j(0)$ converges to $R(0)$ in norm, which proves (i).

 Now if all our measures are nonnegative, then the distribution functions $\nu_k$ converge pointwise to the distribution function
$\nu$ at each point of continuity of $\nu$. Arguing as in the proof of (i), we conclude that the resolvents $R_k(z)$ of $\rH_{\rD,k}$ converge in norm to the resolvent $R(z)$ of $\rH_{\rD}$. In order to prove \eqref{eq:3.3} it suffices to note that the  operators  $\rH_{\rD,k}$, $k\in\N$ and $\rH_\rD$ are nonnegative and their spectra are purely discrete.
\end{proof}

\begin{corollary}\label{cor:6.2}
Let $d\nu$ be a finite signed measure on $(0,1)$ which is singular with respect to the Lebesgue measure and $\rH_\rD$ be the corresponding self-adjoint operator in $L^2(0,1)$. Then there is a sequence of measures $\{d\nu_k\}_{k=1}^\infty$ such that
\be\label{eq:6.7}
 d\nu_k(x) =\sum_{i=1}^{N_k}\gB_{k,i} \delta(x-x_{k,i}), \quad N_k<\infty,
\ee
and the corresponding operators $\rH_{\rD,k}$ converge in the norm resolvent sense to the operator $\rH_\rD$.
\end{corollary}

\begin{proof}
It suffices to note (cf.\ \cite[8.1.6 Example]{boII}) that any finite signed measure $d\nu$ can be approximated in a weak-$*$ sense by measures of the form \eqref{eq:6.7}. Now Theorem \ref{th:approx} (i) completes the proof.
\end{proof}

\begin{remark}
Choose for simplicity $d\nu(x)=-\delta(x)$. It is well known \cite{aghh} (see also \cite{KM_09}) that the corresponding operator has precisely one negative eigenvalue. One can approximate the measure  $d\nu$ in the weak-$*$ sense by nonpositive  absolutely continuous measures $d\nu_k$, i.e., $d\nu_k(x)=p_k(x)dx$. Moreover, by Theorem \ref{th:approx} (i), we can assume that the corresponding operators $\rH_{\rD,k}$ converge to $\rH_\rD$ in the norm resolvent sense.  Notice that the negative spectrum of  $\rH_{\rD,k}$ consists of infinitely many eigenvalues which accumulate at $-\infty$. Since $\sigma(\rH_{\rD,k})\to \sigma(\rH_\rD)$ as $k\to \infty$, all negative eigenvalues of $\rH_{\rD,k}$ go to $-\infty$ as $k\to \infty$.
\end{remark}

Note that on $\R$ we have at least strong resolvent convergence:

\begin{corollary}
Suppose  $(a,b)=\R$ and $\tau$ is in the limit-point case at both endpoints.
Let $d\nu_k$ and $q_k$ be the quantities $d\nu$ and $q$ restricted to $[-k,k]$. Then the corresponding operators $\rH_{\rD,k}$ converge in the strong resolvent sense to the operator $\rH_\rD$.
\end{corollary}

\begin{proof}
Due to our limit-point assumption the functions in $\gD_{\max}\cap L^2_c(\R)$ are a core for $\rH_\rD$.
But for every $f\in\gD_{\max}\cap L^2_c(\R)$ we have $\lim_{k\to\infty} \tau_k f = \tau f$ and the claim follows from \cite[Lemma~6.36]{tschroe}.
\end{proof}

Note that we could also replace $\rH_{\rD,k}$ in the previous corollary by the restriction of $\rH_\rD$ to $[-k,k]$ with a (say) Dirichlet boundary condition at
both endpoints and still have (generalized) strong resolvent convergence. Moreover, combining the previous results shows that we can approximate $\rH_\rD$ by ones with smooth
coefficients with compact support in (generalized) strong resolvent sense.

\section{Spectral asymptotics}\label{sec:asymp}

The main aim of this section is to investigate spectral asymptotics of Hamiltonians with $\delta'$-interactions.
Throughout this section we will always assume that $\tau$ is regular at the left endpoint $a$, that is, $a>-\infty$ and for some $c\in(a,b)$ we have $|d\nu|((a,c))<\infty$ and $q\in L^1(a,c)$. For notational convenience we also set $P(x)=x-a+\nu(x)$ and $\nu(x)=\int_{[a,x)}d\nu(t)$, $x\in (a,b)$.

\subsection{Eigenvalue asymptotics}\label{ss:asymp1}

We begin with the following result, which provides weak eigenvalue asymptotics in the regular case.

\begin{lemma}\label{lem:4.1}
Let $\tau$ be regular, that is, the interval $(a,b)$ is bounded, $|d\nu|$ is a finite measure on $(a,b)$ and $q\in L^1(a,b)$. Then the operator $\rH_\rD$ (cf.\ Section \ref{seq:res}) has purely discrete spectrum and its eigenvalues satisfy
\be\label{eq:4.1A}
N(t) = \begin{cases}
\frac{b-a}{\pi}\sqrt{t}+o(\sqrt{t}), & t\to+\infty, \\
 o(\sqrt{|t|}), & t\to-\infty.
\end{cases}
\ee
Here $N(t)$ denotes the number of eigenvalues of $\rH_\rD$ between zero and $t$.
\end{lemma}

\begin{proof}
To prove the result, it suffices to apply \cite[Theorem 7.3]{ben}. Indeed, since $P(x)=x+\nu(x)$, where $d\nu$ is a singular measure, we immediately obtain \eqref{eq:4.1A}.
\end{proof}

\begin{corollary}\label{cor:4.2A}
If $\tau$ is regular, then the eigenvalues of the operator $\rH_\rD$ satisfy
\be\label{eq:4.1}
\frac{n}{\sqrt{\lambda_n}} \to \begin{cases} \frac{b-a}{\pi}, & n\to +\infty, \\ 0, & n\to - \infty, \end{cases}
\ee
where the second limit is void in the case when $\rH_\rD$ is lower semibounded.
\end{corollary}

\begin{proof}
The claim follows by applying Lemma \ref{lem:4.1} and using the identity
\begin{align}
\lim_{t\to\pm\infty} \frac{N(t)}{\sqrt{t}}=\lim_{n\to\pm\infty}\frac{n}{\sqrt{\lambda_n}}.
\end{align}
\end{proof}

\subsection{Asymptotics of $m$-functions}\label{ss:asymp2}

Let $m$ be the $m$-function corresponding to the Neumann boundary condition at $a$ (for details we refer to \cite{ben} and \cite{measureSL}).

\begin{definition}
Introduce the functions
\begin{align}
\tilde{P}(x) & =\sup_{a\le s\le t\le x}\Big|\int_{[s,t)} dP\Big|, & \tilde{Q}(x) & =\int_a^x|q(t)|dt,
\end{align}
and define $f:\R_+\to \R_+$ as the generalized inverse of the function 
\be\label{eq:4.2}
F(x) =x\cdot \tilde{P}^{-1}(x),\quad x\in (a,b),
\ee
where $\tilde{P}^{-1}$ is the generalized inverse of $\tilde{P}$.
\end{definition}

We start with the magnitude estimate for $m$.

\begin{lemma}
Assume that $\tau$ is regular at $a$ and that $\Sigma=\supp{d\nu}$ is a closed subset of $(a,b)$ of Lebesgue measure zero.
\begin{itemize}
\item[(i)] Fix $z\in \C_+$ and let $c\in (a,b)$ be the largest number such that
\be
\tilde{P}(c)(|z|W(c)+\tilde{Q}(c))\le \frac{1}{5}.
\ee
Then
\be
|m(z)|\le \frac{22}{9}\frac{1}{(c-a)|\im\, z|}.
\ee
\item[(ii)] For all sufficiently large $z\in \C_+$ the following estimate holds true
\be
|m(z)|\le 13\frac{f(|z|)}{|\sin(\arg z)|}.
\ee
\item[(iii)] Assume that there is a constant $A>0$ such that
\be
\int_a^x \tilde{P}(t)^2dt \le A^2\int_a^x P(t)^2dt,\quad x\in (a,b).
\ee
Then
\be
|m(z)|\ge C|\sin(\arg z)|f(|z|)
\ee
for some constant $C>0$ and all sufficiently large $z$.
\end{itemize}
\end{lemma}

\begin{proof}
The proof follows from \cite[Theorem 3.3]{ben}. We only need to notice that $\tilde{P}(x)>0$ for all $x\in (a,b)$ since $P(x)=x+\nu(x)$, where $d\nu$ is singular and its support $\supp{d\nu}=\Sigma=\overline{\Sigma}$ has Lebesgue measure zero.
\end{proof}

Next, we can specify the magnitude estimates by obtaining one term asymptotics for $m$ under additional assumptions on $\nu$. For definitions and properties of regularly varying functions (in the sense of Karamata) we refer to \cite{kor}.

\begin{theorem}\label{th:4.2}
Assume that $P(x)\sim \tilde{P}(x)$ as $x\to a$. If $P$ is a regularly varying function at $x=a$ of order $\alpha\in [0,1]$, then
\be\label{eq:4.3}
m(r\mu)=C_\alpha (-\mu)^{-\frac{\alpha}{1+\alpha}}f(r)(1+o(1)),\quad r\to \infty,
\ee
where
\be\label{eq:4.4}
C_\alpha=\begin{cases}
\alpha^{\frac{1}{1+\alpha}}(1+\alpha)^{\frac{1-\alpha}{1+\alpha}}\frac{\Gamma(\frac{\alpha}{1+\alpha})}{\Gamma(\frac{1}{1+\alpha})},& \alpha\in (0,1],\\
1, & \alpha=0.
\end{cases}
\ee
The estimate holds uniformly for $\mu$ in any compact subset of $\C_+$.

Moreover, the corresponding spectral function $\rho$ satisfies
\be\label{eq:4.5}
\rho(t)= \begin{cases}
\frac{1+\alpha}{\pi}\sin\big(\frac{\pi}{1+\alpha}\big)C_\alpha \, tf(t)(1+o(1)), & \alpha\in (0,1],\\
o(tf(t)), & \alpha=0,
\end{cases}
\quad  t\to+\infty,
\ee
and
\be
\rho(t)=o(tf(t)),\quad  t\to-\infty.
\ee
\end{theorem}

\begin{proof}
The proof follows from \cite[Theorems 4.3 and 7.1]{ben}.
\end{proof}

\begin{corollary}\label{cor:4.3}
If $a\notin \supp{d\nu}$, then
 \be\label{eq:4.6}
 m(z)=\frac{1}{\sqrt{-z}}(1+o(1)),\quad |z|\to\infty,
 \ee
 and the estimate holds uniformly for $z$ in any nonreal sector of $\C_+$.

Moreover, the corresponding spectral function satisfies
 \be\label{eq:4.7}
 \rho(t)=\frac{2}{\pi}\sqrt{t}(1+o(1)),\quad t\to+\infty.
 \ee
 \end{corollary}

 \begin{proof}
 Since $a\notin\supp{d\nu}$, $P(x)=\tilde{P}(x)=x$ on $(a,a+\varepsilon)$ and hence, applying Theorem \ref{th:4.2} with $\alpha=1$ completes the proof.
 \end{proof}

 \begin{corollary}\label{cor:4.4}
 Let $(a,b)=\R_+$ and $\rH$ be the Hamiltonian with $\delta'$-interactions on the set $X=\{x_k\}_{k=1}^\infty$ with $x_k\uparrow +\infty$. Then the Neumann $m$-function and the corresponding spectral function satisfy \eqref{eq:4.6} and \eqref{eq:4.7}, respectively.
 \end{corollary}

 \begin{proof}
 The proof follows from Corollary \ref{cor:4.3} since $(0,x_1)\cap X=\emptyset$.
 \end{proof}

 The next result shows that the asymptotic behavior of the $m$-function at $\infty$ is determined by the asymptotic behavior of a singular measure $d\nu$ at $x=a$.

 \begin{corollary}
 Let $\nu$ be a  regularly varying function at $x=a$ of order $\alpha\in [0,1]$. Then
 the corresponding $m$-function satisfies \eqref{eq:4.3}.
 \end{corollary}

 \begin{proof}
Since $\nu$ varies regularly at $a$, we conclude that (see \cite{kor})
\begin{align}
\nu(x+a)=x^\alpha g(x),
\end{align}
where $g$ varies slowly at $x=0$. In particular,  for any $\varepsilon>0$ we get $x^\varepsilon=o(g(x))$ as $x\to +0$.
Noting that $P(x)=x-a+\nu(x)$, we get
 \begin{align}
 \frac{P(tx+a)}{P(x+a)}=\frac{tx+(tx)^\alpha g(tx)}{x+x^\alpha g(x)}=t^\alpha \frac{g(xt)}{g(x)}\frac{1+(g(tx))^{-1}(tx)^{1-\alpha}}{1+(g(x))^{-1}x^{1-\alpha}}\to t^\alpha,
 \end{align}
 as $x\to +0$. Therefore, $P$ is a regularly varying function of order $\alpha$ and Theorem \ref{th:4.2} completes the proof.
 \end{proof}

 \begin{remark}
 Note that a $\delta$-interaction on a set $\Sigma$ can be described by the following differential expression
 \begin{align}\label{eqnQmea}
  -\frac{d^2}{dx^2}+ dQ(x),
 \end{align}
 where $dQ$ is a signed Borel measure supported on $\Sigma$. A first rigorous treatment of \eqref{eqnQmea} as a quasi-differential expression was done by Savchuk and Shkalikov in \cite{SavShk99}, \cite{SavShk03} (see also \cite{bare}, \cite{measureSL}, \cite{KosMal12b}). Moreover, it was shown in \cite[Theorem 4]{SavShk99} that the eigenvalues of the Dirichlet realization of \eqref{eqnQmea} in $L^2(a,b)$ in the regular case admit the classical Weyl asymptotics
 \begin{align}
 \sqrt{\lambda_n} \sim \frac{\pi}{b-a} n,\quad n\to \infty.
 \end{align}
 Therefore, the eigenvalues of Hamiltonians with $\delta$- and $\delta'$-interactions on $\Sigma$ have the same asymptotic behavior.

However, it follows from \cite{ben} that the corresponding Neumann $m$-function $m_Q$ has the following high energy behavior
 \begin{align}
 m_Q(z)=\frac{1}{\sqrt{-z}}(1+o(1)),\quad |z|\to \infty,
 \end{align}
 in any nonreal sector. This shows that in contrast to the case of a $\delta'$-interaction on $\Sigma$, the leading term of $m_Q$ at high energies does not depend on  $Q$.
 \end{remark}
%%%%%%%%%%%%%%%%%%%%%%%%%%%%%%%%%%%%%%%%%%%%%%%%%%%

\section{Hamiltonians  with discrete spectrum}\label{sec:VIII}

In this section we are going to study the discreteness of the spectrum of Hamiltonians with $\delta'$-interactions. More precisely, we want to extend the results from \cite{KosMal12}, where  Hamiltonians with $\delta'$-interactions on discrete sets were studied, to the case of $\delta'$-interactions on Cantor-type sets.  

Let $\Sig=\Sig(\nu)$ be the (closed) topological support of the measure $d\nu$. 
Throughout this section, we shall assume that $\Sig$ has Lebesgue measure zero, $|\Sig|=0$. 
Furthermore, if it is not stated explicitly, we always assume that $(a,b)=\R_+$ and that $\Sigma$ is unbounded from above.

\subsection{Semibounded Neumann realizations}

Since $\Sig$ is closed, its complement $\Sig^c$ admits a
decomposition
    \begin{equation}\label{8.2}
\Sig^c=\R_+\backslash \Sig= \bigcup^{\infty}_{k=1}\Delta_k,\qquad  \Delta_k=
(a_{k},b_k),\quad \Delta_i\cap\Delta_j=\emptyset,\ \ i\neq j.
    \end{equation}
If $\Sig$ is not discrete, then there is no natural order to
arrange  the component  intervals $\Delta_k$. 
However, on every interval $\Delta_k$ the differential expression $\tau$ coincides with the usual Sturm--Liouville expression $\tau_q=-\frac{d^2}{dx^2}+q(x)$ and we denote by $\rH_{q,k}^N$  the
Neumann realization of $\tau_q$ in $L^2(\Delta_{k})$,
\begin{align}%\label{eq:HkN1}
\begin{split}
&\rH_{q,k}^Nf:=\tau_qf=-f''+q(x)f,\quad f\in\dom{\rH_{q,k}^N},\\
\dom{\rH_{q,k}^N}&=\{f\in W^{2,1}(\Delta_k):\, f'(a_k)=f'(b_k)=0,\
\tau_qf\in L^2(\Delta_k) \}.
\end{split}\label{eq:HkN2}
\end{align}
Since $q\in L^1(\Delta_k)$, the operator  $\rH_{q,k}^N$ is self-adjoint in $L^2(\Delta_k)$. 
Consider the following operator in $L^2(\R_+)$
 \be\label{eq:H^N}
\rH_{\Sig,q}^N:=\bigoplus_{k\in \N} \rH_{q,k}^N, \qquad
\dom{\rH_{\Sig,q}^N} = \bigoplus_{k\in \N}\dom{\rH_{q,k}^N}.
 \ee
Since $|\Sigma|=0$, the operator $\rH_{\Sig,q}^N$ is densely defined in $L^2(\R_+)$ and moreover, it is self-adjoint. Note that the operators $\rH_{q,k}^N$ are lower semibounded in $L^2(\Delta_k)$. The corresponding form is given by 
     \begin{align}\label{3.8A}
\gt_{q,k}[f] = \int_{\Delta_k}|f'|^2 +q(x) |f|^2\,dx,\quad
\dom{\gt_{q,k}}=W^{1,2}(\Delta_{k}).
     \end{align}
With respect to the decomposition $L^2(\R_+)=
\bigoplus_{k=1}^\infty L^2(\Delta_{k})$, introduce the  form 
\be\label{eq:t_Sq}
\gt_{\Sig,q} := \bigoplus^{\infty}_{k=1}\gt_{q,k}.
\ee
This form is lower semibounded (and
hence closed) if and only if the forms $\gt_{q,k}$ have a finite
uniform lower bound, i.e., there is a constant $C>0$ such that
       \begin{equation}\label{3.8B}
\gt_{q,k}[f_k] \ge -C\|f_k\|^2_{L^2(\Delta_{k})}\quad \text{for all}\quad f_k\in W^{1,2}(\Delta_k),\ k\in\N.
   \end{equation}
In particular, the latter holds true if $q$ is lower semibounded on
$\R_+$.  If the quadratic form $\gt_{\Sig,q}$ is lower semibounded, then being a direct sum of closed (semibounded) quadratic
forms, it is also closed. Moreover, the self-adjoint operator associated
with $\gt_{\Sig,q}$ is $\rH_{\Sig,q}^N$ given by \eqref{eq:H^N}. Assuming
 \eqref{3.8B},  we equip $\dom{\gt_{\Sig,q}}$ with the norm
\be
\|f\|^2_{\gH_{\Sig,q}} := \sum_{k=1}^\infty\gt_{q,k}[f] + (C +1)\|f\|^2_{L^2(\R_+)}, \qquad f\in \dom{\gt_{\Sig,q}},
\ee
and denote by $\gH_{\Sig,q}$ the corresponding (energy) Hilbert
space.
%
%
%%%%%%%%%%%%%%%%%%%%%%%%%%%%%%%%%%%%%%%%%%%%%%%%%%%%%%%%%%%%%
\begin{lemma}\label{lem8.1}
Let $q\in L^1_{\loc}(\R_+)$ and assume that the form $\gt_{\Sig,q}$  is lower
semibounded, i.e., \eqref{3.8B} holds. %%%Let also $E_n^c := \cup_{k=1}^n \Delta_k$.
Then:

$(i)$ The series \eqref{eq:t_Sq} converges unconditionally and 
\begin{align}\label{6.5AB}
 \gt_{\Sig,q}[f] =\sum_{k=1}^\infty \gt_{q,k}[f]=\int_{\R_+} |f'|^2+q(x) |f|^2\, \, dx
     \end{align}
for all $ f\in\dom{\gt_{\Sig,q}}$, that is, the sum does not depend on the choice of the order of the component
intervals.

$(ii)$ For any order of the component intervals $\{\Delta_k\}_{k\in
\N}$ the corresponding energy  space $\gH_{\Sig,q}$ is given by
    \begin{equation}\label{8.13}
\gH_{\Sig,q} = \biggl\{f\in W^{1,2}_{\loc}(\Sig^c):
\lim_{n\to\infty} \int_{\Sig_n^c}  |f'|^2  +
q(x) |f|^2\,\, dx  <\infty \biggr\},
    \end{equation}
where $\Sig_n^c := \cup_{k=1}^n \Delta_k$.
  \end{lemma}

  \begin{proof}
(i) %Choose an order of the component intervals. 
 Condition \eqref{3.8B} implies that for any permutation of the series  \eqref{6.5AB} its
sum  does not take values from the interval $(-\infty, a)$ where
$a= -C\|f\|^2_{L^2(\R_+)}$. Therefore, by the Riemann rearrangement theorem, the series
\eqref{6.5AB} converges absolutely.

(ii) By (i)  the limit in \eqref{8.13}  does not depend on the order
of the component intervals $\Delta_k$.
\end{proof}

Our next aim is to consider the form $\gt_{\Sig,q}$ as a perturbation of the form $\gt_{\Sig}:=\gt_{\Sig,0}$,
\begin{align}\label{6.5B}
 \gt_{\Sig}[f] =\sum_{k=1}^\infty \int_{\Delta_k}|f'|^2\, dx=\int_{\R_+} |f'|^2\, dx,\qquad
 \dom{\gt_{\Sig}}
 =W^{1,2}(\Sig^c),
    \end{align}
where we use the following notation
    \begin{equation}\label{3.2pr}
 W^{1,2}(\Sig^c) =  W^{1,2}(\R_+\backslash \Sig) :=
\bigoplus^{\infty}_{k=1}W^{1,2}(\Delta_k).
  \end{equation}
Note that $\gt_\Sig$ is nonnegative and closed, and hence by Lemma \ref{lem8.1} the definition \eqref{3.2pr} does not depend on the order of the component intervals. 

Next,  consider the following quadratic form in $L^2(\R_+)$
\begin{equation}\label{q_form}
\gq[f]:=\int_{\R_+} q(x)|f|^2\, dx, \quad \dom{\gq} =\{f\in
L^2(\R_+):\ |\gq[f]|<\infty\}.
\end{equation}
Note that the form $\gq$  is lower semibounded (and hence closed) if and only if so
is $q$ on $\R_+$. Define the form $\gt^0_{\Sig,q}$ as a sum of forms $\gt_{\Sig}$ and $\gq$:
   \begin{align}\label{6.5BA}
 &\gt_{\Sig,q}^0[f] := \gt_{\Sig}[f]+\gq[f]  =  \int_{\R_+}|f'|^2 + q(x) |f|^2\, dx,\ \ f\in \dom{\gt_{\Sig,q}^0},\\
 &\dom{\gt_{\Sig,q}^0}= W^{1,2}(\Sig^c;q):=\dom{\gt_\Sig}\cap \dom{\gq},\label{6.5C}
    \end{align}
and note that $\gt_{\Sig,0}^0=\gt_\Sig$. Moreover, if the form $\gt_{\Sig,q}$ defined by \eqref{eq:t_Sq} is lower semibounded, then by definition, $W^{1,2}(\Sig^c;q)\subseteq \gH_{\Sig,q}$. This, in particular, means that $\gt_{\Sig,q}$ is a closed semibounded extension of the form $\gt_{\Sig,q}^0$. On the other hand, if $q$ is nonnegative, then $\gt_{\Sig,q}^0$ is closed as a sum of two nonnegative closed forms. Moreover, in this case $\gt_{\Sig,q}^0=\gt_{\Sig,q}$ and, in particular, $ \gH_{\Sig,q}=W^{1,2}(\Sig^c;q)$.

\begin{corollary}
Assume the conditions of Lemma \ref{lem8.1}. If $\gH_{\Sig,q}\neq  W^{1,2}(\Sig^c;q)$, then $W^{1,2}(\Sig^c;q)$
 is dense  in  $\gH_{\Sig,q}$ and, moreover, forms a first category set.
\end{corollary}

\begin{proof}
The statement immediately follows from the closed graph
theorem.
\end{proof}

The next result provides necessary and sufficient conditions for the form $\gq$ to be $\gt_{\Sig}$-bounded.
%
%%%%%%%%%%%%%%%%%%%%%%%%%%%%%%%%%%%%%%%%%%%%%%%%%%%%%%%%%%%%%%%%%
  \begin{lemma}\label{lem_semibound1}
Let $d_k := |\Delta_k|$, $k\in\N$ and  
\be\label{eq:d*}
\gd^*:= \gd^*(\Sig):=\sup_k\gd_k<\infty.
\ee
 Assume also that $q\in L^1_{\loc}(\R_+)$ and %the  condition
   \begin{equation}\label{I_brinckNew}
\sup_{k\in\N} \frac{1}{d_k}\int_{\Delta_k}|q(x)| dx < \infty.
\qquad\ \
  \end{equation}
 Then:

\begin{itemize}

\item[(i)] The form $\gq$ given by \eqref{q_form}  
is infinitesimally $\gt_{\Sig}$-bounded.

\item[(ii)]  The form $\gt_{\Sig,q}^0$ is lower semibounded and  closed. Moreover, the
equality $\gH_{\Sig,q} = W^{1,2}(\Sig^c)$
holds algebraically and topologically.
\end{itemize}
  \end{lemma}
%%%%%%%%%%%%%%%%%%%%%%%%%%%%%%%%%%%%%%%%%%%%%%%%%%%%%%
The proof follows literally the proof of Lemma 2.7 from \cite{KosMal12} and we omit it.

Next we set
\be\label{eq:q_pm}
 q_{\pm}(x):=(|q(x)|\pm  q(x))/2,\quad x\in\R_+.
\ee

%%%%%%%%%%%%%%%%%%%%%%%%%%%%%%%%%%%%%%%%%%%%%
  \begin{lemma}\label{lem_semibound}
Let $q\in L^1_{\loc}(\R_+)$ and \eqref{eq:d*} be satisfied. If
   \begin{equation}\label{I_brinck}
C_0:=\sup_{k\in\N} \frac{1}{d_k}\int_{\Delta_k}q_-(x) dx < \infty,
\qquad\ \
  \end{equation}
 then:
\begin{itemize}
\item[(i)] The form $\gq_-$ is infinitesimally $\gt_{\Sig}$-bounded
and hence the form $\gt_{\Sig,q}^0$ is lower semibounded and closed.
\item[(ii)] The following equalities
  \begin{equation}\label{eq:dom=w}
\gH_{\Sig,q} = W^{1,2}(\Sig^c; q) = W^{1,2}(\Sig^c; q_+) = \gH_{\Sig; q_+}
  \end{equation}
hold algebraically and topologically and the operator associated
with
$\gt_{\Sig,q}$ coincides with  $\rH_{\Sig,q}^N = (\rH_{\Sig,q}^N)^*$.     %%%$\dom(\gt_{X,\gB,q})=\dom(\gt_{X,\gB^+,q^+})$.
\item[(iii)]
 If, additionally, condition   \eqref{I_brinckNew} is  satisfied with
 $q_+$  in place of  $|q|$, then  \eqref{I_brinck}
is also necessary for the form $\gt_{\Sig,q}^0$ to be lower semibounded.
In particular, condition \eqref{I_brinck}  is  necessary for lower
semiboundedness  whenever $q$  is  negative.
   \end{itemize}
  \end{lemma}
%%%%%%%%%%%%%%%%%%%%%%%%%%%%%%%%%%%%%%%%
\begin{proof}
(i) and (ii) immediately follow from Lemmas \ref{lem8.1} and
\ref{lem_semibound1} and the KLMN theorem.

(iii) Set  $h_k(x)=\frac{1}{\sqrt{\gd_k}}\chi_{\Delta_k}(x)$. Since $h_k\in W^{1,2}(\Sig^c; q)$, we
get
\begin{align}
\gt_{\Sig,q}[h_k]= \frac{1}{\gd_k}\int_{\Delta_k}q(x)dx \ge
-C\|h_k\|^2_{L^2}=-C \end{align} for all $k\in\N$. Noting that $q=q_+-q_-$, we finally obtain
\begin{align}
-\frac{1}{\gd_k}\int_{\Delta_k}q_-(x)dx  \ge -C
-\frac{1}{\gd_k}\int_{\Delta_k}q_+(x)dx \ge -\tilde{C}, \quad k\in\N,
\end{align}
which implies  \eqref{I_brinck}.
   \end{proof}
   
   \begin{remark}
   The results of this subsection remain true if 
   \begin{align}
   q\in L^1(\Delta_k)\quad \text{for all}\quad k\in\N.
   \end{align}
   Note that this condition is weaker than the condition $q\in L^1_{\loc}(\R_+)$ if $\Sig$ is not a discrete subset of $\R_+$.
   \end{remark}
%%%%%%%%%%%%%%%%%%%%%%%%%%%%%%%%%%%%%%%%%%%

  \subsection{Discreteness of the spectrum of $\rH_{\Sig,q}^N$}  \label{ss:8.2}
  The main result of this subsection is the following discreteness criterion
  for the operator $\rH_{\Sig,q}^N$ defined by \eqref{eq:HkN2}--\eqref{eq:H^N}.

%%%%%%%%%%%%%%%%%%%%%%%%%%%%%%%%%%%%%%%%%%%%%
  \begin{theorem}\label{th:discr_N}
  Let $q\in L^1_{\loc}(\R_+)$, $\Sigma$ be unbounded with $|\Sig|=0$ and \eqref{eq:d*} be satisfied. Assume also that the operator $\rH_{\Sig,q}^N$
  given by \eqref{eq:HkN2}--\eqref{eq:H^N} is lower semibounded.
  If the potential $q$ satisfies \eqref{I_brinck}, then the spectrum of $\rH_{\Sig,q}^N$
  is purely discrete if and only if 
  \begin{itemize}
  \item[(i)]  $q$ satisfies Molchanov's condition 
   \begin{equation}\label{4.1_Molchanov}
\lim_{x\to\infty}\int^{x + \varepsilon}_x q(t)dt=+ \infty\quad \text{for every}\ \varepsilon>0, % \qquad \text{as}\qquad x\to \infty.
    \end{equation}
  \item[(ii)] \be\label{eq:4.9}
  \lim_{k\to\infty}\frac{1}{\gd_k}\int_{\Delta_k}q(x)dx= +\infty.
  \ee
  \end{itemize}
  \end{theorem}
%%%%%%%%%%%%%%%%%%%%%%%%%%%%%%%%%%%%%%%%%%%%%%%
      \begin{proof} 
Let us prove 
\emph{sufficiency} first. By Lemma
\ref{lem_semibound}, the form $\gt_{\Sig,q}^0$ given by
\eqref{6.5BA}--\eqref{6.5C} is lower
 semibounded and closed in $ L^2(\R_+)$.
 Moreover, $\gt_{\Sig,q}^0=\gt_{\Sig,q}$ and the corresponding energy space $\gH_{\Sig,q}$ coincides (algebraically and topologically) with $ W^{1,2}(\Sig^c; q) = W^{1,2}(\Sig^c; q_+)$. By the Rellich theorem,
it suffices to show that the embedding 
\[
i_{\Sig,q}:\gH_{\Sig; q}\hookrightarrow L^2(\R_+)
\] 
is compact, i.e.,  the unit ball  $\mathbb{U}_{\Sig, q}:= \{f\in \gH_{\Sig; q}: \
\|f\|_{\gH_{\Sig; q}}\le 1 \}$ is relatively compact in $L^2({\R}_+)$. 
Clearly, it suffices to prove sufficiency for nonnegative
potentials.  Hence without loss of generality we can assume that $q\ge 1$. 

 Fix  $\varepsilon>0$. Using \eqref{4.1_Molchanov} and \eqref{eq:4.9}, we can find  $p:=p(\varepsilon)\in\N$ such that
     \begin{equation}\label{3.6}
\frac{1}{\gd_k} \int_{\Delta_k}q(t) dt  > \frac{1}{\varepsilon} \quad \text{and} \quad \int_x^{x+\varepsilon}q(t)dt>1     \end{equation}
for all $k>p$ and $x>x_p:=\max\{ t: t\in \cup_{k=1}^p\Delta_k$\}, respectively.  Next we  set
   \begin{equation}\label{8.18}
\N_p'(\varepsilon):=\{k\in\N: \, k\ge p,\ |d_k|\le \varepsilon\},\quad  \N_p''(\varepsilon):=\{k\in\N: \, k\ge p,\ |d_k|> \varepsilon\}.
  \end{equation}
Clearly, $\N_p'(\varepsilon) \cup  \N_p''(\varepsilon)=\{k\in\N:\ k\ge  p\}$.  
Arguing as in the proof of Theorem 3.4 from \cite{KosMal12}, we get (cf. the estimates (3.14) and (3.20) in \cite{KosMal12})
\begin{align}\label{3.7}
%\int_{\Delta_k}|f|^2 dx
\|f\|^2_{L^2(\Delta_k)}\le 2\varepsilon {\int_{\Delta_k} q(x) |f|^2\, dx} +  2\varepsilon^2
\|f\|^2_{W^{1,2}(\Delta_k)}, \qquad k\in \N'_{p}({\varepsilon}),
          \end{align}
 and 
      \begin{align}\label{3.12}
\|f\|^2_{L^2(\Delta_k)}   \le 8\varepsilon {\int_{\Delta_k} q(x) |f|^2\, dx} + 
8\varepsilon\|f\|^2_{W^{1,2}(\Delta_k)},\qquad k\in \N''_{p}(\varepsilon).
      \end{align}
Summing up  \eqref{3.7}  and   \eqref{3.12}, we finally get
     \begin{align}\begin{split}
\sum_{k=p}^\infty\|f\|^2_{L^2(\Delta_k)} \le&
8\varepsilon\sum^\infty_{k=p} \Big(\int_{\Delta_{k}} q(x) |f|^2\, dx +
 \|f\|^2_{W^{1,2}(\Delta_k)} \Big) \\
 &\le 8C_1 \varepsilon\|f\|^2_{W^{1,2}(\Sigma^c;q)} \le 8C_1
\varepsilon,\quad f\in \mathbb{U}_{\Sig, q}. \label{3.13}
    \end{split} \end{align}

Now notice that the embedding $W^{1,2}(\Sig_n^c)\hookrightarrow L^2(\Sig_n^c)$, where $n\in
\N$ and $\Sig_n^c = \cup_{k=1}^n \Delta_k$, is compact since $W^{1,2}(\Delta_k)$ is compactly embedded into 
$L^2(\Delta_k)$ for each  $k\in \N$. 
Hence \eqref{3.13} implies that the set $W^{1,2}(\Sig_p^c )$ forms a
compact $\varepsilon$-net for the set $i(\mathbb{U}_{\Sig, q})$ in
$L^2(\R_+)$. Therefore, the set $i(U_{\Sig,q})$ is compact in $L^2(\R_+)$
and  the embedding $i: W^{1,2}(\Sig^c; q)\hookrightarrow
L^2(\R_+)$ is compact too.  

{\em Necessity.} The proof literally follows the proof of Theorem 3.4 in \cite{KosMal12}.
  \end{proof}

        \begin{corollary}\label{cor8.4}
Let $q\in L^1_{\loc}(\R_+)$ satisfy \eqref{I_brinck}.  If $\gd_k \to 0$, then
the spectrum of $\rH_{\Sig,q}^N$ is purely  discrete if and only if $q$ satisfies condition \eqref{eq:4.9}.
  \end{corollary}

  \begin{proof}
By Theorem \ref{th:discr_N},  it suffices to show that
\eqref{eq:4.9} implies \eqref{4.1_Molchanov} if $\gd_k\to 0$. Since
$q$ satisfies   \eqref{I_brinck}, we can restrict ourselves to the
case of a nonnegative $q$, $q=q_+$.
By \eqref{eq:4.9}, for any $N\in\N$ there exists
$p_1 = p_1(N)$ such that

    \begin{equation}\label{3.30}
\frac{1}{\gd_k}\int_{\Delta_k} q(x)dx > N, \qquad  k\ge p_1.
   \end{equation}
Fix $\varepsilon>0$. Since $d_k \to 0$, there exists $p_2=
p_2(\varepsilon)$ such that
    \begin{equation}\label{3.31}
d_k\le \frac{\varepsilon}{3},\qquad   k\ge p_2.
  \end{equation}
Let $p :=\max(p_1, p_2)$ and  let $x_p:=\max \cup_{k=1}^p\Delta_k$.
Using \eqref{3.30},  \eqref{3.31} and the  non-negativity of $q$, we get
\begin{align}
\int^{x+ \varepsilon}_x q(t)dt \ge \sum_{k:\, \Delta_k\subseteq [x,x+\varepsilon]} \int_{\Delta_k} q(t)dt  \ge \sum_{k:\, \Delta_k\subseteq [x,x+\varepsilon]} Nd_k\ge
N\Big(\varepsilon - \frac{2\varepsilon}{3}\Big) =
N\frac{\varepsilon}{3}
\end{align}
for all $x>x_p$. 
The latter implies that $q$ satisfies Molchanov's condition
\eqref{4.1_Molchanov} since $N$ is arbitrary.
\end{proof}

  \begin{remark}%\marginpar{\tiny FIXME: Additional condition $c_{n+1}-d_n\to 0$!}
Condition $\lim_{k\to\infty}d_k=0$ is satisfied if, for instance, $\Sig
=\cup^{\infty}_{j=1}\Sig_j$, where  $\Sig_j$ are Cantor type sets on disjoint
intervals $[r_{2j},r_{2j+1}]$ such that  $\lim_j(r_{j+1} - r_j)= 0$. % and $\lim_j(c_{j+1} - d_j)= 0$.
  \end{remark}

One can consider the operator $\rH_{\Sig,q}^N$ on any
finite interval $(a,b)$ instead of $\R_+$. Arguing as in the proof
of Theorem \ref{th:discr_N} and  Corollary \ref{cor8.4} and noting
that $\gd_k \to 0$ as $k\to \infty$
since $(a,b)$ is a finite interval, one arrives at the following statement.

   \begin{corollary}\label{rem:8.10}
Assume that $q\in L^1_{\loc}(a,b)$ satisfies \eqref{I_brinck}.
Then the spectrum of the Neumann realization $\rH_{\Sig,q}^N$ is
discrete if and only if condition \eqref{eq:4.9} holds.

If additionally  $q$ is nonnegative and continuous as a map
from $(a,b)$ to $\R_+\cup\{\infty\}$, then $\rH_{\Sig,q}^N$ has
discrete spectrum if and only if
  \begin{equation}\label{8.5}
q(x)=\infty\qquad \text{for all}\quad  x\in \Sig.
      \end{equation}
\end{corollary}

\subsection{Necessary conditions for the discreteness of the spectrum of $\rH_{\nu,q}$}\label{ss:8.3}

We begin by presenting some necessary conditions for the Hamiltonians $\rH_{\nu,q}$ to be lower semibounded.

        \begin{proposition}\label{prop4.5}
Let $|\Sig| =0$ and  \eqref{eq:d*} be satisfied.  Assume also that $q\in L^1_{\loc}(\R_+)$ satisfies
 \eqref{I_brinck} and the operator $\rH_{\nu,q}$ is self-adjoint and lower semibounded in $L^2(\R_+)$. If $q$ does not satisfy Molchanov's condition \eqref{4.1_Molchanov}, then the spectrum of $\rH_{\nu,q}$ is  not
discrete.
    \end{proposition}

     \begin{proof}
 Let $\rH_q$ be the Neumann realization of $-d^2/dx^2+q(x)$ in $L^2(\R_+)$. Since $q$ satisfies \eqref{I_brinck}, we know that $\rH_q$ is lower semibounded and self-adjoint in $L^2(\R_+)$. Moreover, the corresponding form $\gt_q$ is given by (see, e.g., \cite{AKM_10})
 \be
 \gt_q[f]=\int_{\R_+}|f'|^2+ q(x) |f|^2\, dx,\quad \dom{\gt_q}=\gH_q:=W^{1,2}(\R_+)\cap \gt_q.
 \ee     
     
By assumption, the operator $\rH_{\nu,q}$ is lower semibounded and self-adjoint. Without loss of generality we can assume that $\rH_{\nu,q}\ge I$. Denote by $\gt$ and $\gH$ the corresponding quadratic form and the energy space. Let us show that $\gH_q$ is continuously embedded into $\gH$. First of all, note that if $f\in W^{1,2}_{\loc}(\R_+)$, then $f^{\qd}(x)=f'(x)$ for almost all $x\in \Sig^c$ and $f^{\qd}(x)=0$ for $d\nu$-almost all $x\in\Sig$ since 
\begin{align}
f(x)-f(0)=\int_0^x f'\, dt=\int_0^x f^{\qd}\, dP(t)=\int_0^x f'\, dt+\int_0^x f^{\qd}\, d\nu(t)
\end{align}
for all $x\in\R_+$. Therefore, $f\in W^{1,2}_{\loc}(\R_+;dP)$ and we get  (cf. section \ref{ss:qform})
\begin{align}
\gt[f]=\int_{\R_+}|f^{\qd}|^2dx+\int_{\R_+} q(x) |f|^2\, dx=\int_{\R_+}|f'|^2+ q(x) |f|^2\, dx=\gt_q[f]<\infty
\end{align}
for all $f\in W^{1,2}_c(\R_+)$. This implies that $W^{1,2}_c(\R_+)=\gH_q\cap W^{1,2}_c(\R_+)\subset \gH$. It remains to note that $\gH$ is closed with respect to the energy norm $\gt$ and $W^{1,2}_c(\R_+)$ is a core for $\gt_q$. Therefore, the closure of $W^{1,2}_c(\R_+)$ in $\gH$ coincides with $\gH_q$ since $\|\cdot\|_{\gH}^2=\gt[\cdot]$ on $\gH$ and hence $\gH_q\subset \gH$. Moreover, in view of \cite[Theorem 2.6.2]{Ios65} (see also \cite[Remark IV.1.5]{Kato66}), the embedding $i_1: \gH_q
\hookrightarrow \gH$ is continuous.

Finally, if the spectrum  $\sigma(\rH_{\nu,q})$ is discrete, then,
by the Rellich theorem, the embedding $i_2: \gH
 \hookrightarrow L^2(\R_+)$ is compact. Hence so is the embedding
 $i = i_2 i_1:  \gH_q\to L^2(\R_+)$. However, this implies that Molchanov's condition
\eqref{4.1_Molchanov} is satisfied. This contradiction completes the
proof.
     \end{proof}

\begin{remark}\label{Rem3.8}
Note that Proposition \ref{prop4.5} is no longer true if $\Sig$ is nowhere dense but $|\Sig|>0$.
        \end{remark}
        
        As an immediate corollary of Proposition \ref{prop4.5} we obtain the following result.

        \begin{corollary}\label{cor3.8}
Let $q\in L^\infty(\R_+)$ and $|\Sig| =0$. If the operator
$\rH_{\nu,q}$ is lower semibounded, then its spectrum is not
discrete. In particular, the spectrum of $\rH_{\nu} = \rH_{\nu,0}$ is not discrete
whenever it is lower semibounded.
    \end{corollary}

\begin{proof}
If $q\in L^\infty(\R_+)$, then it satisfies \eqref{I_brinck} but does
not satisfy \eqref{4.1_Molchanov}. Proposition \ref{prop4.5}
completes the proof.
  \end{proof}

The next result states that a condition similar to \eqref{eq:4.9} is also necessary for the discreteness. Note that we do not assume that $|\Sigma|=0$  in this case.

     \begin{proposition}\label{prop_Second_Necces_Cond}
Let $q\in L^1_{\loc}(a,b)$ be nonnegative.  Let also $\Sig\subset (a,b)$ be the (closed) topological support of $d\nu$ and $\Sig^c=\cup_{k\in\N}\Delta_k$, where $\Delta_i\cap \Delta_j=\emptyset$ whenever $i\neq j$. If the operator  $\rH_{\nu,q}$ is lower semibounded and its spectrum is discrete,
then
    \begin{equation}\label{8.48}
\frac{1}{d^2_k} + \frac{1}{d_k} \int_{\Delta_k}q(x) dx \to
\infty\quad \text{as}\quad  k\to\infty.
    \end{equation}
    \end{proposition}
%%%%%%%%%%%%%%%%%%%%%%%%%%%%%%%%%%%%
\begin{proof}
Let $\gd_k=|\Delta_k|=b_k-a_k$ and set

   \begin{equation}
h_k(x)=\frac{1}{2\sqrt{\gd_k}}\, \chi_{\Delta_k}(x)\left(1-\frac{2}{\gd_k}\Big|x-\frac{\gd_k}{2}\Big|\right),\quad x\in (a,b).
   \end{equation}
Clearly,  $h_k\in W^{1,2}_c((a,b);dP)\cap W^{1,2}_c(a,b)$ and 
   \begin{equation}\label{8.58}
\|h_k\|_{L^2(a,b)}= \frac{1}{12},\qquad \|h_k'\|_{L^2(a,b)}= \frac{1}{\gd_k^2},\quad k\in\N.
   \end{equation}
Moreover,
    \be\label{8.51}
\gt[h_k]=\frac{1}{d^2_k} +
\frac{1}{4\gd_k}\int_{\Delta_k} \Big(1-\frac{2}{\gd_k}\Big|x-\frac{\gd_k}{2}\Big|\Big)^2  q(x) dx\le \frac{1}{d^2_k} +
\frac{1}{4\gd_k}\int_{\Delta_k}q(x)dx.
    \ee
If condition \eqref{8.48} is not satisfied, then there is
a subsequence $\{k_j\}^{\infty}_{j=1}$ such that
    \begin{equation}\label{8.52}
\frac{1}{d^2_{k_j}} + \frac{1}{d_{k_j}} \int_{\Delta_{k_j}}q(x) dx
\le C_0<\infty, \qquad j\in\N.
   \end{equation}
This immediately implies that the  subsequence $\{h_{k_j}\}_{j=1}^\infty$ is bounded in the energy space $\gH$. However, all these functions are uniformly bounded in $L^2(a,b)$ and have disjoint supports. Hence this sequence is not compact in $L^2(a,b)$.  Therefore, the
embedding $\gH\hookrightarrow L^2(a,b)$ is not
compact and hence, by the Rellich theorem, the operator $\rH_{\nu,q}$
is not discrete. This contradiction completes the proof.
   \end{proof}

  \subsection{Sufficient conditions for the discreteness of the spectrum of $\rH_{\nu,q}$}  \label{ss:8.4}

Consider the form \eqref{eqnform} introduced in Section \ref{ss:qform}
\be\label{eq:qform0}
\gt_{\nu,q}^0[f]=\int_{\R_+}|f^{\qd}|^2dP(x)+\int_{\R_+} q(x) |f|^2\, dx,\quad f\in\dom{\gt_{\nu,q}^0}=W^{1,2}_c(\R_+; dP).
\ee
Since $|\Sig|=0$, we have $f^{\qd}(x) = f'(x)$ a.e.\ on $\R_+$ for all 
 $f\in W^{1,2}_c(\R_+; dP)$. Moreover, $W^{1,2}_c(\R_+; dP)\subset W^{1,2}(\Sig^c;q)$ and hence the quadratic form
\eqref{eq:qform0} admits the following representation
\begin{align}\begin{split}
\gt_{\nu,q}^0[f]
=  \int_{\R_+}|f'|^2 dx &+
\int_{\Sig}|f^{\qd}|^2\,dP(x)+ \int_{\R_+}q(x) |f|^2\, dx  \\
&= \gt_{\Sig,q}[f] + 
\int_{\Sig}|f^{\qd}|^2\,d\nu,\quad f\in W^{1,2}_c(\R_+;dP).\label{4.10_form}
    \end{split}\end{align}
 Therefore, the form \eqref{eq:qform0} can be considered as an additive perturbation of the form $\gt_{\Sig,q}$.  Note that by Lemma \ref{lem:3.2} the form $\gt_{\nu,q}^0$ is always closable if it is lower semibounded. With  a nonnegative measure $d\nu$ we  associate the following form
\be\label{eq:nuform}
\nu[f]=\int_{\R_+} |f^{\qd}|^2d\nu(x),\qquad f\in W^{1,2}(\R_+;dP).
\ee

\begin{lemma}\label{lem8.10}
Let the form $\gt_{\Sig,q}$ be lower semi-bounded and $d\nu$ be nonnegative, $d\nu=
d\nu_+$. Then the form
\be\label{eq:8.37}
\gt_{\nu,q}[f]=\gt_{\Sig,q}[f]+ \nu[f],\qquad \dom{\gt_{\nu,q}}=\dom{\gt_{\Sig,q}}\cap \dom{\nu}
\ee
is lower semibounded and closed. The corresponding energy  space  $\gH_{\nu,q}$ is given by
\be\label{eq:8.38}
\gH_{\nu,q} = \biggl\{f\in \dom{\gt_{\nu,q}}:  
\lim_{n\to\infty} \int_{\Sig^c_n}  |f'|^2 + |f|^2
 q(x)dx + \nu[f] <\infty \biggr\}.     
\ee
\end{lemma}

\begin{proof}
Without loss of generality we can assume that $q\ge 1$ on $\R_+$. Let us show that every Cauchy sequence $\{f_n\}_{n=1}^\infty\subset \gH_{\nu,q}$ is convergent and has a unique limit. Since $\gH_{\Sig,q}$ and $W^{1,2}(\R;dP)$ are Hilbert spaces, there are functions $f_h\in\gH_{\Sig,q}$ and $f_P\in W^{1,2}(\R_+;dP)$ such that $f_n\to f_h$ and $f_n\to f_P$ in $\gH_{\Sig,q}$ and $W^{1,2}(\R_+;dP)$, respectively. Clearly, $f_h=f_P$ for all $x\in\Sig^c=\cup_{k\in\N}\Delta_k$. To show that $f_P$ is defined uniquely by its restriction to $\Sigma^c$
assume that $f_P= 0$ on $\Sigma^c$. Since $f\in W^{1,2}(\R_+,dP)$,Êthe values $f(x \pm)$Êexist for each $x\in\R_+$ 
and $f(x+) \neq f(x-)$ for at mostÊa countable subset. As $\Sigma^c$ is dense in $\R_+$, one has $f(x+)Ê= f(x-) = 0$ for all $x$.

The last claim immediately follows from Lemma \ref{lem8.1}.
\end{proof}

\begin{corollary}\label{lem_Embedding}
Let the form $\gt_{\Sig,q}$ be lower semibounded and  $d\nu$ be a nonnegative singular measure with $|\Sig|=0$.
Then the norms $\|\cdot\|_{\gH_{\Sig,q}}$ and $\|\cdot\|_{\gH_{\nu,q}}$ are compatible and the map
        \begin{equation}\label{8.11}
i: \gH_{\nu,q} \hookrightarrow \gH_{\Sig,q}, \qquad
i(f):=f\lceil \Sig^c,
        \end{equation}
is a well-defined continuous embedding.
  \end{corollary}

   \begin{proof}
It follows from  Lemma \ref{lem8.10} that the 
restriction  $f\lceil \Sig^c_n$ is well-defined for all $f\in\gH_{\nu,q}$ since $\gH_{\nu,q}\subseteq \gH_{\Sig,q}$.
Since the measure $d\nu$ is nonnegative, one has   $\|f\|_{\gH_{\Sig,q}} \le \|f\|_{\gH_{\nu,q}}$ for all $f\in {\gH_{\nu,q}}$. Hence the embedding \eqref{8.11} is continuous.
      \end{proof}

Combining Lemma \ref{lem_semibound}  with Lemma \ref{lem8.10}
we arrive at the following result.

   \begin{corollary}\label{cor:8.12}
Assume that conditions \eqref{eq:d*} and \eqref{I_brinck} are satisfied  and the measure
$d\nu$ is nonnegative, $d\nu=d\nu_+$. Then the form $\gt_{\nu,q}$ given by \eqref{eq:8.37} is
lower semibounded and closed. 
Moreover, the energy space $\gH_{\nu,q}$ is given by \eqref{eq:8.38}.
   \end{corollary}

The next result provides a sufficient condition for the discreteness of the spectra of Hamiltonians $\rH_{\nu, q}$.

\begin{theorem}\label{th8.2}
Let $d\nu$ be  nonnegative and such that $|\Sig|=0$ and \eqref{eq:d*} hold. Assume also that $q\in
L^1_{\loc}(\R_+)$ satisfies  \eqref{I_brinck}. Then
conditions \eqref{4.1_Molchanov} and \eqref{eq:4.9} are sufficient for the  Hamiltonian  $\rH_{\nu,q} (= \rH_{\nu,q}^*)$ to be lower semibounded and to have a discrete spectrum.
   \end{theorem}
%%%%%%%%%%%%%%%%%%%%%%%%%%%%%%%%%%%%%%%%%%%%%%%%%%%%%%%
   \begin{proof}
By Lemma \ref{lem_semibound}, the form $\gt_{\Sig,q}$ is lower semibounded and, moreover, by Theorem \ref{th:discr_N}, the energy space $\gH_{\Sig,q}$ is compactly embedded into $L^2(\R_+)$. 
On the other hand, by Corollary \ref{lem_Embedding}, $\gH_{\nu,q}$ is continuously embedded into $\gH_{\Sig,q}$. 
Therefore, $\gH_{\nu,q}$ is compactly embedded into $L^2(\R_+)$ and by the Rellich theorem, the Hamiltonian $\rH_{\nu,q}$ has discrete spectrum. 
   \end{proof}

We can easily obtain the analog of Theorem \ref{th8.2} in the case of a finite interval.

   \begin{corollary}\label{cor8.14}
Let $(a,b)$ be a finite interval and $q\in L^1_{\loc}(a,b)$ satisfy \eqref{I_brinck}. Assume also that $d\nu$ is a nonnegative measure on $(a,b)$ with $|\Sig|=0$. Then   $\rH_{\nu, q}$ is lower
semibounded  and its spectrum is discrete whenever condition
\eqref{eq:4.9} is satisfied. 
   \end{corollary}

      \begin{example}
      Let $(a,b)= (0,1)$ and $d\nu$ be a nonnegative finite measure with $|\Sig|=0$. Then $\Sig^c = (0,1)\backslash \Sig= \cup_{k=1}^{\infty}
\Delta_k$, where $\Delta_k = (a_k,b_k)$, $k\in \N$. Define $q:\Sig^c\to\R_+$ by
\begin{equation}
q(x)=\sum_{k=1}^\infty\chi_{\Delta_k}(x)\frac{c_k}{\sqrt{(x-a_k)(b_k-x)}},\qquad x\in\Sig^c,
\end{equation}
where $\{c_k\}_{k=1}^\infty \in \ell^1(\N)$ is a sequence of positive numbers. Note that 
    \begin{equation}
\|q\|_{L^1(\Delta_k)}=\pi c_k,\qquad \|q\|_{L^1(0,1)}=\pi\sum_{k=1}^\infty c_k<\infty.
    \end{equation}
Therefore, the operator $\rH_{\nu,q}$ has discrete spectrum since both endpoints are regular. 
However, by Corollary \ref{rem:8.10}, the operator $\rH_{\Sig,q}^N$ has purely discrete spectrum if and only if 
    \begin{equation}
\frac{1}{d_k} \int_{\Delta_k}q(x) dx  =
\pi\frac{c_k}{\gd_k}\to\infty\quad \text{as}\quad
k\to\infty.
    \end{equation}
\end{example}
%%%%%%%%%%%%%%%%%%%%%%%%%%%%%%%%%%%%%%%%
  \begin{remark}
It is interesting to compare Corollary \ref{cor8.14} with Theorem
\ref{th:sa} and Corollary \ref{cor:2.3} in the case $-\infty<a<b <
+\infty$. If $q\in L^{\infty}(a,b)$, then, in view of Corollary \ref{rem:8.10},
the Neumann realization $\rH^N_{\Sig,q}$ is not discrete although, by
Corollary \ref{cor:2.3}, $\rH_{\nu,q}$ is. This comparison shows that
in this case the perturbation $\nu[\cdot]$ (see \eqref{eq:nuform}) of the
form $\gt_{\Sig,q}$ does not preserve the essential spectrum of $\rH_{\Sig,q}^N$.
  \end{remark}

By Proposition \ref{prop4.5}, Molchanov's condition is not only sufficient but is also necessary for the discreteness of the spectrum of the operator $\rH_{\nu,q}$. Our next aim is to show that condition
\eqref{eq:4.9} is not necessary for the discreteness. Namely, we are going to show that there are cases when the spectrum of $\rH_{\nu,q}$  is discrete, however, the essential spectrum of the Neumann realization $\rH_{\Sig,q}^N$ might be nontrivial. 

We begin with the following simple auxiliary lemma.
 
  \begin{lemma}\label{embed_lem}
Let $d\nu$ be a nonnegative finite measure on $(0,a)$. Then ${W}^{1,2}((0,a);dP)$ is compactly embedded into $L^2(0,a)$.
  \end{lemma}

    \begin{proof}
Each $f\in {W}^{1,2}\bigl((0,a);dP\bigr)$ is a function of bounded variation and hence $f\in L^2(0,a)$. Moreover, using the representation \eqref{eq:1_qd'} and the Cauchy--Schwarz inequality, we get
\begin{align}\begin{split}
|f(x)|\le |f(0)|+\int_0^x|f^{\qd}|dP&\le |f(0)|+ \sqrt{P(x)}\|f^{\qd}\|_{L^2((0,x);dP)}\\
&\le |f(0)|+ \sqrt{P(a)}\|f\|^2_{W^{1,2}((0,a);dP)}.
\end{split}\end{align}
This estimates shows that the embedding is continuous. Compactness follows from the discreteness of spectra of Hamiltonians with two regular endpoints.
    \end{proof}

   \begin{remark}
If $\nu$  is singular continuous, then every $f\in W^{1,2}((0,a);dP)$ is continuous. Moreover, the Arzel\'a--Ascoli theorem shows that the unit ball in $W^{1,2}((0,a);dP)$ is compact in $C[0,a]$. Therefore,
the embedding  $W^{1,2}\bigl((0,a);dP\bigr)\hookrightarrow
L^2(0,a)$  is continuous and compact.
      \end{remark}

   \begin{theorem}\label{prop8.18}
Let $d\nu$ be a nonnegative measure on $\R_+$ with $|\Sig|=0$ such that
\be\label{eq:8.44}
0\le\nu_n:=d\nu([n-1,n])\le C_\nu<\infty,\qquad n\in\N,
\ee
for some constant $C_\nu>0$ independent of $n\in\N$. 
Let  $q$ satisfy conditions \eqref{I_brinck}
and \eqref{4.1_Molchanov}.  If $\Sig^c=\Sig_1^c\cup \Sig_2^c$, where $\Sig_1^c=\cup_{k\in\N} \Delta_k'$  and $\Sig_2^c=\cup_{k\in\N} \Delta_k''$ are such that
\begin{equation}\label{8.56}
\sum_{k=1}^\infty \gd_k'<\infty\qquad \text{and}\qquad \lim_{k\to\infty}\frac{1}{\gd_k''}\int_{\Delta_k''}q(x)dx=\infty, 
\end{equation}
where $\gd_k'=|\Delta_k'|$ and $\gd_k''=|\Delta_k''|$, $k\in\N$,
then the operator $\rH_{\nu,q}$ is lower semibounded and its spectrum is discrete.
    \end{theorem}

    \begin{proof}
  By Lemma \ref{lem8.10}, the form $\gt_{\nu,q}$ is lower semibounded and closed. 
Without loss of generality we assume that $q\ge 1$ on $\R_+$. Let us show that the unit ball
 $\mathbb{U}_{\nu, q}:= \{f\in \gH_{\nu,q}: \ \gt_{\nu,q}[f]\le 1\}$ of $\gH_{\nu,q}$
 is compact in $L^2({\R}_+)$. 
Let  $\varepsilon>0$. As in the proof of Theorem \ref{th:discr_N},
we obtain the estimates   \eqref{3.7} and
\eqref{3.12} for the set $\Sig_2^c$. It remains to evaluate the integrals
$\int_{\Delta_{k}'}|f(x)|^2dx$ over the intervals $\Delta_k'$.
Note that (cf. the proof of Lemma \ref{embed_lem})
     \begin{equation}\label{8.59A}
\frac{1}{2}|f(x)|^2\le |f(y)|^2 + (1 + \nu_n)
\int^n_{n-1}|f^{[1]}(t)|^2 dP(t),
     \end{equation}
for all $x,y\in [n-1,n]$.
Since $f$ is locally of bounded variation on $\R_+$, there exists $x_n\in[n-1,n]$ such that
    \begin{equation}
|f(x_n)|^2 \le  \int^n_{n-1}|f(t)|^2 dt,  \qquad n\in \N.
    \end{equation}
Combining this relation with \eqref{8.59A} and using \eqref{eq:8.44}, we get
    \begin{align}\label{eq:8.47}
    \begin{split}
\frac{1}{2}|f(x)|^2 &\le \|f\|^2_{L^2(n-1,n)} + (1+\nu_n)\|f^\qd\|^2_{L^2([n-1,n];dP)}\\
&\le \|f\|^2_{L^2(\R_+)} + (1+C_\nu)\|f^\qd\|^2_{L^2(\R_+;dP)}\le 1+C_\nu,
\end{split}
 \end{align}
 which holds for all $x\in \R_+$ and $f\in \mathbb{U}_{\nu, q}$. 
 Next, in view of the first condition in \eqref{8.56}, there is $p_3\in\N$ such that $\sum_{k\ge p_3}\gd_k'< \varepsilon $. Therefore, \eqref{eq:8.47} implies
   \begin{equation}\label{eq:8.48}
\sum^{\infty}_{k = p_3} \|f\|^2_{L^2(\Delta_k')}%=\sum^{\infty}_{k = p_3}\int_{\Delta_k'}|f|^2dx 
\le 2(1+C_{\nu}) \sum^{\infty}_{k = p_3}\gd'_k < 2\varepsilon(1+C_{\nu}), \quad f\in
\mathbb{U}_{\nu, q}.
   \end{equation}
Setting $p = \max\{p', p'', p_3\}$ and  combining \eqref{eq:8.48}
with \eqref{3.7} and \eqref{3.12} we get
  \begin{equation}\label{8.63}
\sum^{\infty}_{k = p} \|f\|^2_{L^2(\Delta_k')}  +  \sum^{\infty}_{k = p} \|f\|^2_{L^2(\Delta_k'')} < \tilde{C} \varepsilon, \qquad f\in
\mathbb{U}_{\nu, q},
     \end{equation}
showing that the "tails" of functions $f\in \mathbb{U}_{\nu, q}$ are
uniformly small.

Let $a=a(\varepsilon):=\max\{x:x\in\cup_{k=1}^{p-1} (\Delta_k'\cup \Delta_k'')\}$. By Lemma  \ref{embed_lem},  
$W^{1,2}((0,a);dP)$ is compactly embedded into $ L^2(0,a)$. 
On the other hand,  since $|\Sig|=0$, % and $(a,+\infty)\subseteq (\cup_{k=1}^p \Delta_k')\cup(\cup_{k=1}^p \Delta_k'')$, 
 \eqref{8.63} implies % that
  \begin{equation}\label{8.66A}
\|f\|^2_{L^2(a,+\infty)} \le \sum^{\infty}_{k = p} \|f\|^2_{L^2(\Delta_k')}  +  \sum^{\infty}_{k = p} \|f\|^2_{L^2(\Delta_k'')} < \tilde{C} \varepsilon, \quad f\in \mathbb{U}_{\nu, q}.
     \end{equation}

Denoting by  $\mathbb{U}_{\nu, q}^a$ the unit ball in
$W^{1,2}((0,a);dP)$ and using \eqref{8.66A},
we conclude that $i(\mathbb{U}_{\nu, q}^a)$ forms a compact
$\varepsilon$-net for the set $i(\mathbb{U}_{\nu, q})$ in
$L^2(\R_+)$. Hence the set $i(\mathbb{U}_{\nu, q})$  is compact in
$L^2(\R_+)$. The Rellich theorem completes the proof.
   \end{proof}

\begin{remark}
Theorem  \ref{prop8.18} is also valid for discrete
measures $d\nu$ and, in particular, it is valid for operators with
$\delta'$-point interactions (see Example \ref{ex:delta'}). Note that for Hamiltonians with $\delta'$-point interactions
a different sufficient condition guaranteeing that the operator
$H_{\nu,q}$ has  discrete spectrum while the Neumann realization
$\rH_{\Sig,q}^N$ has not, is contained in \cite[Proposition
3.10]{KosMal12}.
   \end{remark}

\begin{corollary}\label{cor8.21}
Let $q\in L^1_{\loc}(\R_+)$ satisfy 
\eqref{I_brinck}.
Assume that $\{r_{n}\}_{n\in \N}\subset\R_+$ is a strictly increasing sequence such that
   \begin{equation}\label{8.64}
\sum^{\infty}_{n=1}|I_{2k}|<\infty,\qquad I_k:=[r_{k-1},r_{k}],\ \ |I_k|=r_{k}-r_{k-1}.
  \end{equation}
Let also  $d\nu$ be nonnegative and such that $|\Sig|=0$, $\{r_k\}_{k=1}^\infty\subseteq \Sig$, $\Sig \subset \cup_{k\in\N} I_{2k}$ and 
\begin{align}
\sup_{k\in\N}d\nu(I_{2k})=C_\nu<\infty.
\end{align}
Then the spectrum of $\rH_{\nu,q}$ is discrete
if   $q$ satisfies Molchanov's condition \eqref{4.1_Molchanov} and
  \begin{equation}\label{8.66}
\lim_{k\to\infty}\frac{1}{|I_{2k+1}|}\int_{I_{2k+1}}q(x)dx = \infty.
  \end{equation}
   \end{corollary}

  \begin{proof}
Noting that $I_{2k}\backslash \Sig=\cup^{\infty}_{j=1}\Delta_{j,
k}$, where $\Delta_{j, k}$ are pairwise disjoint, we set $\{\Delta_k'\}^{\infty}_{k=1}:= \{\Delta_{j,
n}\}^{\infty}_{j,n=1}$, and $\Delta_k''=I_{2k+1}$. Clearly, the first condition in \eqref{8.56} is
implied by \eqref{8.64}. Moreover, the second one is equivalent to \eqref{8.66}.
It remains to apply Theorem \ref{prop8.18}.
   \end{proof}

   \begin{corollary}\label{cor8.22}
Assume the conditions of Corollary \ref{cor8.21}. Let also 
   \begin{equation}\label{8.64A}
 0<\varepsilon_0 = \inf_{n} |I_{2n+1}| \le  \sup_{n} |I_{2n+1}| = \varepsilon_1<\infty.
     \end{equation}
 Then the spectrum of $\rH_{\nu,q}$ is discrete if and only if
$q$ satisfies condition \eqref{4.1_Molchanov}.
   \end{corollary}

  \begin{proof}
Note that \eqref{4.1_Molchanov} is necessary for the discreteness in view of Proposition \ref{prop4.5}.
Therefore, by Corollary \ref{cor8.21}, it suffices to show that $q$ satisfies \eqref{8.66}. However, it immediately follows from Molchanov's condition \eqref{4.1_Molchanov} and \eqref{8.64A}.
   \end{proof}

As an immediate corollary we obtain the following result.

\begin{corollary}\label{8.28}
Assume the
conditions of Corollary \ref{cor8.22}. Then the spectrum of $\rH_{\nu,q}$ is discrete if

      \begin{equation}\label{8.54}
 \lim_{x\to \infty}q(x) = +\infty.
     \end{equation}
   \end{corollary}

  \begin{remark}
If $\Sig$ is a discrete
set, i.e. $\rH_{\nu,q}$ is the Hamiltonian  with $\delta'$-point
interactions, then condition \eqref{8.54} (without  additional  assumption
\eqref{8.64A}) is sufficient for the discreteness of $\rH_{\nu,q}$ (see \cite[Corollary 3.8]{KosMal12}). However, if
$\Sig$ is not discrete, then condition \eqref{8.54} does
not guarantee the discreteness of $\rH_{\nu,q}$.
  \end{remark}


\begin{thebibliography}{XX}

\bibitem{aghh}
S.\ Albeverio, F.\ Gesztesy, R.\ H{\o}egh-Krohn, and H.\ Holden, {\em Solvable Models in Quantum Mechanics}, 2nd ed.,
AMS Chelsea Publishing, Providence, RI, 2005.

\bibitem{AKM_10}
S.\ Albeverio, A.\ Kostenko, and M.\ Malamud,
{\em Spectral theory of semi-bounded  Sturm--Liouville  operators with local interactions on a discrete set}, J.\ Math.\ Phys.\ \textbf{51} (2010), 102102, 24 pp.

\bibitem{AKMN13}
S.\ Albeverio, A.\ Kostenko,  M.\ Malamud and H.\ Neidhardt,
{\em Spherical Schr\"odinger operators with $\delta$-type interactions}, J.\ Math.\ Phys.\ \textbf{54} (2013), 052103, 24 pp.

\bibitem{albniz03}
S. \ Albeverio and L.\ Nizhnik, {\em On the number of negative eigenvalues of a one-dimensional Schr\"odinger operator with point interactions}, Lett.\ Math.\ Phys.\ {\bf 65} (2003), 27--35.

\bibitem{albniz}
S.\ Albeverio and L.\ Nizhnik, {\em A Schr\"odinger operator with a $\delta'$-interaction on a Cantor set and Krein--Feller operators}, Math.\ Nachr.\ {\bf 279} (2006), no.~5-6, 467--476.

\bibitem{BehLanLot12}
J.\ Behrndt, M.\ Langer, and V.\ Lotoreichik, {\em Schr\"odinger operators with $\delta$ and $\delta'$--potentials supported on hypersurfaces}, Ann.\ Henri Poincar\'e {\bf 14} (2013), 385--423.

\bibitem{ben}
C.\ Bennewitz, {\em Spectral asymptotics for Sturm--Liouville equations},
Proc.\ London Math.\ Soc.\ (3) {\bf 59} (1989), no.~2, 294--338.

\bibitem{bare}
A.\ Ben Amor and C.\ Remling,  {\em Direct and inverse spectral theory of one-dimensional Schr\"odinger operators with measures},
Int.\ Equat.\ Oper.\ Theory {\bf 52} (2005), no.~3, 395--417.

\bibitem{bo}
V.\ I.\ Bogachev, {\em Measure Theory. I}, Springer-Verlag, Berlin, Heidelberg, 2007.

\bibitem{boII}
V.\ I.\ Bogachev, {\em Measure Theory. II}, Springer-Verlag, Berlin, Heidelberg, 2007.

\bibitem{braniz}
J.\ F.\ Brasche and L.\ Nizhnik, {\em One-dimensional Schr\"odinger operators with $\delta'$-interactions on a set of Lebesgue measure zero}, Oper.\ Matrices {\bf 7} (2013), 887--904. 

\bibitem{bsw}
D.\ Buschmann, G.\ Stolz, and J.\ Weidmann, {\em One-dimensional Schr\"odinger operators with local point interactions}, J.\ reine Angew.\ Math. \textbf{467} (1995), 169--186.

\bibitem{cs}
T.\ Cheon and T.\ Shigehara, {\em  Realizing discontinuous wave functions with renormalized short--range
potentials},  Phys. Lett. A {\bf 243} (1998), 111--116.

\bibitem{TimeScales}
J.\ Eckhardt and G.\ Teschl, {\em On the connection between the Hilger and Radon--Nikodym derivatives}, J.\ Math.\ Anal.\ Appl.\ {\bf 385} (2012), 1184--1189.

\bibitem{measureSL}
J.\ Eckhardt and G.\ Teschl, {\em Sturm--Liouville operators with measure-valued coefficients}, J. d'Analyse Math. {\bf 120} (2013), 151--224.

\bibitem{EGNT}
J.\ Eckhardt, F.\ Gesztesy, R.\ Nichols, and G.\ Teschl, {\em Weyl--Titchmarsh theory for Sturm--Liouville operators with distributional potentials}, Opuscula Math.\ {\bf 33} (2013), 467--563.

\bibitem{EGNT2}
J.\ Eckhardt, F.\ Gesztesy, R.\ Nichols, and G.\ Teschl, {\em Supersymmetry and Schr\"odinger-type operators with distributional matrix-valued potentials}, J.\ Spectr.\ Theory (to appear), \arxiv{1206.4966}.

\bibitem{EGNT3}
J.\ Eckhardt, F.\ Gesztesy, R.\ Nichols, and G.\ Teschl, {\em Inverse spectral theory for Sturm--Liouville operators with distributional potentials}, J.\ Lond.\ Math.\ Soc.\ (2) {\bf 88} (2013), 801--828.

\bibitem{enz}
P.\ Exner, H.\ Neidhardt, and V.\ Zagrebnov, {\em Potential approximations to $\delta'$: an inverse Klauder
phenomenon with norm-resolvent convergence}, Comm. Math. Phys. {\bf 224} (2001), 593--612.

\bibitem{fle}
A.\ Fleige, {\em Spectral Theory of Indefinite Krein--Feller Differential Operators}, Math.\ Research {\bf 98}, Akademie Verlag, Berlin, 1996.

\bibitem{gh}
F.\ Gesztesy and H.\ Holden, {\em A new class of solvable models in quantum mechanics describing
point interactions on the line}, J.\ Phys.\ A: Math.\ Gen.\ {\bf 20} (1987), 5157--5177.

\bibitem{GolOri10}
N.\ I.\ Goloshchapova and L.\ L.\ Oridoroga, {\em On the negative spectrum  of  one-dimensional Schr\"odinger operators  with point interactions}, Int.\ Equat.\ Oper.\ Theory {\bf 67} (2010), 1--14.

\bibitem{GolHry}
Yu.\ D.\ Golovaty and R.\ O.\ Hryniv, {\em Norm resolvent convergence of singularly scaled
Schr\"odinger operators and $\delta'$-potentials}, Proc. Roy. Soc. Edinburgh A {\bf 143} (2013), 791--816.

\bibitem{GolMan}
Yu.\ D.\ Golovaty and S.\ S.\ Man'ko, {\em Solvable models for the Schr\"odinger operators with $\delta'$-like
potentials}, Ukr. Math. Bull. {\bf 6} (2009), no.2, 179--212.

\bibitem{ghm}
A.\ Grossmann, R.\ Hoegh-Krohn, and M.\ Mebkhout, {\em The one--particle theory of periodic point
interactions}, J.\ Math.\ Phys.\ {\bf 21} (1980), 2376--2385.

\bibitem{hrymyk12}
R.\ O.\ Hryniv and Ya.\ V.\ Mykytyuk, {\em Self-adjointness of Schr\"odinger operators with singular potentials}, Methods Funct.\ Anal.\ Topology {\bf 18} (2012), 152--159.

\bibitem{kac}
I.\ S.\ Kac, {\em The existence of spectral functions of generalized second order differential systems with boundary conditions at the singular end},
Amer.\ Math.\ Soc.\ Transl.\ Ser.\ 2, {\bf 62} (1967), 204--262.

\bibitem{kakr74}
I.\ S.\ Kac and M.\ G.\ Kre\u{\i}n, {\em On the spectral functions of the string}, Amer.\ Math.\ Soc.\ Transl.\ Ser.\ 2, {\bf 103} (1974), 19--102.

\bibitem{Kato66}
T.\ Kato, \emph{Perturbation Theory for Linear Operators}, 2nd ed.,
Springer-Verlag, Berlin-Heidelberg, New York, 1966.

\bibitem{kor}
J.\ Korevaar, {\em Tauberian Theory: A Century of Developments}, Grundlehren der Mathematischen Wissenschaften, Springer, 2004.

\bibitem{KM_09}
A.\ Kostenko and M.\ Malamud,  {\em 1--D Schr\"odinger operators with local point interactions on a discrete set}, J.\ Differential Equations {\bf 249} (2010), 253--304.

\bibitem{KosMal10}
A.\ Kostenko and M.\ Malamud,  {\em  Schr\"odinger operators with $\delta'$-interactions and the Krein--Stieltjes string}, Doklady Math.\ {\bf 81} (2010), no.~3,  342--347.

\bibitem{KosMal12b}
A.\ Kostenko and M.\ Malamud,  {\em 1--D  Schr\"odinger operators with
local point interactions: a review}, in "Spectral Analysis, Integrable Systems, and Ordinary Differential Equations",
H.\ Holden et al.\ (eds), 235--262, Proc. Symp.  Pure Math. {\bf 87}, Amer.\ Math.\ Soc., Providence, 2013.

\bibitem{KosMal12}
A.\ Kostenko and M.\ Malamud,  {\em  Spectral theory of semibounded Schr\"odinger operators with $\delta'$-interactions}, Ann.\ Henri Poincar\'e {\bf 15} (2014), 501--541.

\bibitem{kur}
P.\ Kurasov, {\em Distribution theory for discontinuous test functions and differential
operators with generalized coefficients}, J.\ Math.\ Anal.\ Appl.\ {\bf 201} (1996), 297--323.

\bibitem{ming}
A.\ B.\ Mingarelli, {\em Volterra-Stieltjes integral equations and generalized ordinary differential expressions},
Lecture Notes in Mathematics {\bf 989}, Springer, Berlin, 1983.

\bibitem{niz}
L.\ P.\ Nizhnik, {\em Schr\"odinger operator with $\delta'$-interaction}, Funct.\ Anal.\ Appl.\ {\bf 37} (2003), no.~1, 85--88.

\bibitem{ogu}
O.\ Ogurisu, {\em On the number of negative eigenvalues of a Schr\"odinger operator with point interactions}, Lett.\ Math.\ Phys.\ {\bf 85} (2008), 129--133.

\bibitem{SavShk99}
A.\ M.\ Savchuk and A.\ A.\ Shkalikov, {\em Sturm--Liouville operators with singular potentials}, Math.\ Notes \textbf{66}  (1999), no. 5-6, 741--753.

\bibitem{SavShk03}
A.\ M.\ Savchuk and A.\ A.\ Shkalikov, {\em Sturm--Liouville operators with distribution potentials}, Trans.\ Moscow Math.\ Soc.\ (2003), 143--190.

\bibitem{tschroe}
G.\ Teschl, {\em Mathematical Methods in Quantum Mechanics; With Applications to Schr\"odinger Operators},
Graduate Studies in Mathematics {\bf 99}, Amer. Math. Soc., Providence, RI, 2009.

\bibitem{volkmer2005}
H.\ Volkmer, {\em Eigenvalue problems of Atkinson, Feller and Krein, and their mutual relationship}, Electron.\ J.\ Differential Equations {\bf 2005}, No.~48, 15 pp.

\bibitem{Ios65}
K.\ Yosida, {\em Functional Analysis}, Springer, Berlin (1980).

\end{thebibliography}
\end{document}